\documentclass{conm-p-l}
\usepackage{amssymb,tikz}
\usepackage{verbatim}
\usepackage{graphicx}
\usepackage[cmtip,all]{xy}
\usepackage{amsmath}
\usepackage{amsfonts,amssymb,amsthm}
\usepackage{times}
\usepackage{amscd}
\usepackage{epsfig}
\usepackage{float}

\usepackage{pb-diagram}

\usepackage{srcltx}
\usepackage{hyperref}

\newtheorem{theorem}{Theorem}[section]
\newtheorem{lemma}[theorem]{Lemma}

\newtheorem{cor}[theorem]{Corollary}

\newtheorem{prop}[theorem]{Proposition}

\theoremstyle{definition}
\newtheorem{definition}[theorem]{Definition}

\theoremstyle{remark}

\numberwithin{equation}{section}

\newtheorem*{thm*}{Theorem}

\newtheorem*{thmA}{Theorem A}
\newtheorem*{thmB}{Theorem B}

\newcommand{\ca}{\mathrm{CA}}

\newcommand{\C}{\mathbb{C}}

\newcommand{\disk}{\mathbb{D}}

\newcommand{\R}{\mathbb{R}}
\newcommand{\Z}{{\mathbb{Z}}}

\newcommand{\ol}{\overline}

\newcommand{\sm}{\setminus}

\newcommand{\qml}{\mathrm{QML}}

\newcommand{\lam}{\mathcal{L}}
\newcommand{\hlam}{\widehat{\mathcal{L}}}

\newcommand{\ch}{\mathrm{CH}}

\newcommand{\si}{\sigma}
\newcommand{\0}{\varnothing}
\newcommand{\uc}{\mathbb{S}}

\newcommand{\M}{\mathcal{M}}

\renewcommand{\le}{\leqslant}
\renewcommand{\ge}{\geqslant}

\def\ofl{\mathrm{OL}}

\begin{document}

\title[Dynamical generation of parameter laminations]
{Dynamical generation of parameter laminations}

\author{Alexander~Blokh}

\address[Alexander~Blokh]
{Department of Mathematics\\ University of Alabama at Birmingham\\
Birmingham, AL 35294}

\email{ablokh@math.uab.edu}

\author{Lex Oversteegen}

\address[Lex Oversteegen]
{Department of Mathematics\\ University of Alabama at Birmingham\\
Birmingham, AL 35294}

\email{overstee@uab.edu}

\author{Vladlen~Timorin}

\address[V.~Timorin]{Faculty of Mathematics, HSE University\\
6 Usacheva St., 119048 Moscow, Russia}

\email{vtimorin@hse.ru}



\date{February 23, 2019; in the revised form May 3 2019}


\dedicatory{Dedicated to the memory of Sergiy Kolyada}

\begin{abstract}
Local similarity between the Mandelbrot set and quadratic Julia sets
manifests itself in a variety of ways. We discuss a combinatorial one,
in the language of geodesic laminations. More precisely, we compare
quadratic invariant laminations representing Julia sets with the
so-called Quadratic Minor Lamination (QML) representing a locally
connected model of the Mandelbrot set. Similarly to the construction of
an invariant lamination by pullbacks of certain leaves, we describe
how QML can be generated by properly understood pullbacks of certain
minors. In particular, we show that the minors of all
non-renormalizable quadratic laminations
can be obtained by taking limits of ``pullbacks'' of minors from the main cardioid.
\end{abstract}

\maketitle

\section*{Introduction}\label{s:intro}
Quadratic polynomials $P_c(z)=z^2+c$, where $c\in\C$, play an important role in complex dynamics.
They provide a simple but highly non-trivial example of polynomial dynamical systems
(note that every quadratic polynomial is affinely conjugate to one of the form $P_c$), and
this family is universal in the sense that many properties of the $c$-parameter plane
  reappear locally in almost any analytic family of holomorphic maps \cite{Mc00}.
The central object in the $c$-plane is the \emph{Mandelbrot set} $\M_2$.
By definition, $c\in\M_2$ if the Julia set $J(P_c)$ of $P_c$ is connected, equivalently,
  if the sequence of iterates $P^n_c(c)$ does not escape to infinity (see \cite{hubbdoua85}).

The Mandelbrot set is compact and connected.
It is not known if it is locally connected, but
there is a nice model $\M^c_2$, due to Douady, Hubbard and Thurston, of
$\M_2$ (i.e., there exists a continuous map $\pi:\M_2\to \M^c_2$
such that point inverses are connected); moreover, if $\M_2$ is locally connected,
$\pi$ is a homeomorphism.
Namely, set $\disk=\{z\in\C: |z|<1\}$ and $\uc=\{z\in\C\,|\,|z|=1\}$;
call $\disk$ the \emph{unit disk} and $\uc$ the \emph{unit circle}.
There are pairwise disjoint chords (including degenerate chords, i.e.
singletons in $\uc$) or polygons inscribed in
  $\ol\disk=\{z\in\C: |z|\le 1\}$ such that,
  after collapsing all these chords and polygons to points, we get a quotient space $\M^c_2$.
We will write $\qml$ for the set consisting of all these chords and edges of all these polygons.
This set is called the \emph{quadratic minor lamination}.

More generally, a (geodesic) \emph{lamination} is a set of chords (called \emph{leaves}) in $\ol\disk$
  that contains all points of $\uc$ such that the limit of any converging sequence of leaves is a leaf.
The lamination $\qml$ can be described explicitly.
For example, one can algorithmically generate countably many leaves dense in $\qml$,
and there are several known constructions, e.g. \cite{lav86, la89}
(other combinatorial viewpoints on $\M_2^c$ and $\qml$ can be found in
\cite{bopt16, kel00, pr08, sch09}).
In this paper, a new construction is provided that is based on taking preimages under the angle doubling map.
Each of the sets $\M_2$ and $\M^c_2$ contains countable and dense family of homeomorphic copies of itself.
Thus, $\M_2$ and $\M^c_2$ are examples of so-called \emph{fractal} sets.

\begin{figure}[!htb]
    \centering
        \includegraphics[height=0.15\textheight]{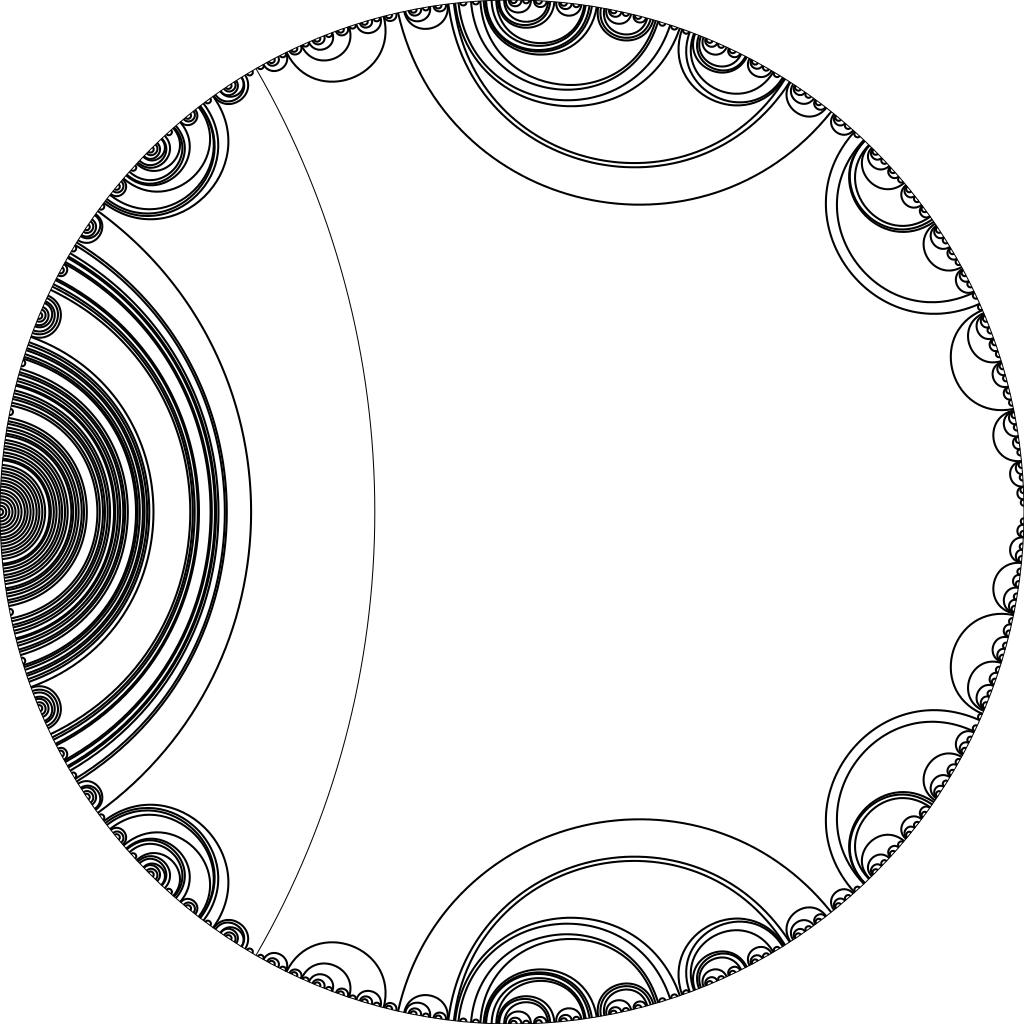}
        \caption{The geolamination $\qml$}
        \label{fig:qml}
\end{figure}

A description of $\qml$ by Thurston \cite{thu85} refers to laminational models of Julia sets.
By the \emph{filled Julia set} $K(P_c)$ of a polynomial $P_c$ we mean the set of points $z\in\C$ with $P_c^n(z)\not\to\infty$.
The \emph{Julia set} $J(P_c)$ is the boundary of $K(P_c)$.
If $K(P_c)$ is locally connected, then
  it can be also obtained from $\ol\disk$ by collapsing leaves and finite polygons of some lamination $\lam(P_c)$.

Indeed, if $K(P_c)$ is locally connected, the Riemann map defined
for the complement of $K(P_c)$ can be extended onto $\uc$
which gives rise to a continuous map $\psi: \uc\to J(P_c)$ that
semiconjugates the angle doubling map $\si_2:\uc\to \uc$ (taking
$z\in\uc$ to $z^2$) and $P_c|_{J(P_c)}$. Considering convex hulls of
\emph{fibers} (point-inverses) of $\psi$ and collecting boundary edges
of these convex hulls, we obtain the lamination $\lam(P_c)$.
Declaring points $x, y$ of $\uc$ \emph{equivalent} if and only if
$\psi(x)=\psi(y)$ we arrive at the \emph{invariant laminational
equivalence $\sim_c$} and the associated quotient space
$J_{\sim_c}$ of $\uc$ (the \emph{topological Julia set}), homeomorphic
to $J(P_c)$. Equivalence classes of $\sim_c$ have pairwise disjoint
convex hulls. The \emph{topological polynomial}
$f_{\sim_c}:J_{\sim_c}\to J_{\sim_c}$, induced by $\si_2$, is
topologically conjugate to $P_c|_{J(P_c)}$. Laminational equivalence
relations $\sim$ similar to $\sim_c$ can be introduced with no references to polynomials by
listing their properties similar to those of $\sim_c$
(this can be done for degrees higher than $2$ as well). In that
case one also considers the 
collection $\lam_\sim$ of the edges of convex
hulls of all $\sim$-classes and all singletons in $\uc$ called the \emph{q-lamination
(generated by $\sim$)}.

A lamination $\lam_{\sim_c}$ thus obtained satisfies certain dynamical properties
(in our presentation we rely upon \cite{bmov13}). Below we think of $\si_2$ applied
to a chord $\ell$ with endpoints $a$ and $b$ so that it
maps to the chord whose endpoints are $\si_2(a)$ and $\si_2(b)$;
we can think of this as an extension of $\si_2$ over $\ell$ and make it linear on $\ell$.
The properties are as follows:
\begin{enumerate}
  \item \textbf{forward invariance:} for every $\ell\in\lam$, we have $\si_2(\ell)\in\lam$;
  \item \textbf{backward invariance:} for every $\ell\in\lam$ we have $\ell=\si_2(\ell_1)$ for some $\ell_1\in\lam$;
  \item \textbf{sibling property:} for every $\ell\in\lam$, we have $-\ell\in\lam$.
\end{enumerate}
Here $-\ell$ is the image of $\ell$ under the
map $z\mapsto -z$ of $\uc$.
(Under this map all angles are incremented by $\frac 12$ modulo $1$).
The leaf $-\ell$ is called the \emph{sibling} of $\ell$.
A chord which is a diameter of $\disk$ is said to be \emph{critical}.
Laminations with properties $(1)$--$(3)$ are called \emph{quadratic invariant laminations}.
By \cite{bmov13} all quadratic q-laminations $\lam_\sim$ are invariant, however the converse is not true and
there are quadratic invariant laminations that are not q-laminations. Below we often call
quadratic invariant laminations simply \emph{quadratic laminations}.

Properties (1) -- (3) from above deal exclusively with leaves. To
understand the dynamics one also considers components of the complement
in $\ol\disk$ to the union of all leaves of $\lam$. More precisely, a
\emph{gap} of $\lam$ is the closure of a component of
$\disk\sm\bigcup_{\ell \in\lam}\ell$. Gaps $G$ are said to be
\emph{finite} or \emph{infinite} according to whether $G\cap\uc$ is a
finite or infinite set. By \cite{bmov13} if $G$ is a gap of a quadratic
lamination $\lam$, then either all its edges map  to one leaf of
$\lam$, or all its edges map to a single point in $\uc$, or the convex hull of the set $\si_2(G\cap
\uc)$ is a gap of $\lam$ which one can view as the \emph{image} of $G$. Moreover,
the map on the boundary of $G$ satisfies \textbf{gap invariance}:
either there exists a critical edge of $G$, or the map
$\tau=\si_2|_{G\cap \uc}$ extends to $\uc$ as an orientation preserving
covering map $\hat \tau$ such that $G\cap \uc$ is the full preimage of
$\tau(G\cap \uc)$ under $\hat \tau$. Gap invariance was part of the
original definition of a (geodesic) lamination given by Thurston in
\cite{thu85}. It allows us to extend the map $\si_2$ onto the entire
$\ol\disk$ if a quadratic lamination $\lam$ is given. Indeed, we have
already described how $\si_2$ acts on leaves; it can then be extended
over gaps using the \emph{barycentric} construction (see \cite{thu85}
for details). 

Due to the backward invariance property, quadratic laminations can often be generated by taking pullbacks of leaves.
By a \emph{pullback} of a leaf $\ell\in\lam$, we mean a leaf $\ell_1\in\lam$ such that $\si_2(\ell_1)=\ell$.
An \emph{iterated pullback} of $\ell$ of level $n$ is defined as a leaf $\ell_n\in\lam$ with $\si_2^n(\ell_n)=\ell$.
The concept of (iterated) pullback is widely used in the study of (quadratic) invariant laminations. In this paper we show
that it can also be used as one studies \emph{parameter} laminations, i.e., laminations which do not satisfy conditions
(1) --- (3), such as $\qml$.
Let us now discuss $\qml$ in more detail.

To measure arc lengths on $\uc$, we use the normalized Lebesgue measure
(the total length of $\uc$ is $1$). The
\emph{length of a chord} is by definition the length of the shorter
circle arc connecting its endpoints. Following Thurston, define a
\emph{major leaf} (a \emph{major}) of a quadratic lamination as a
longest leaf of it. (There may be one longest leaf that is critical or
two longest leaves that are siblings.)
The \emph{minor leaf} (the \emph{minor}) of a lamination is the $\si_2$-image of a major.
If a minor $m$ is non-periodic, then there exists a unique maximal lamination with minor $m$ denoted by $\lam(m)$.
If a minor $m$ is periodic and non-degenerate, then we define $\lam(m)$ as the unique \emph{q-lamination} with minor $m$.
Finally, if $m$ is a periodic singleton, then we explicitly define $\lam(m)$ later in the paper
 so that $m$ is the minor of $\lam(m)$ (note, that in this case the choice of $\lam(m)$ is irrelevant for our purposes).
Call $\lam(m)$ the \emph{minor leaf lamination associated with $m$}. Observe that
there are no minors that are non-degenerate and have exactly one
periodic endpoint.

A chord in $\ol\disk$ with endpoints $a$ and $b$ is denoted by $ab$. If two distinct chords intersect
in $\disk$, we say that they \emph{cross} or that they are \emph{linked}.
Given a chord $ab$, without a lamination, we have ambiguity in defining pullbacks of $ab$.
Namely, there are two preimages of $a$ and two preimages of $b$, and, in general,
 there are several ways of connecting the preimages of $a$ with the preimages of $b$.
Even if we prohibit crossings and impose the sibling property, then there are three ways
(two ways of connecting the preimages by two chords and one way of connecting them by four chords).
However, if we know that the pullbacks must belong to $\lam(m)$, then they are well defined.
We can describe the process of taking pullbacks explicitly, without referring to $\lam(m)$.
One of the main objectives of this paper is to apply a similar pullback construction to $\qml$.

Thurston's definition of $\qml$ is simply the following: $\qml$ consists precisely of the minors of all quadratic laminations.
In particular, it is true (although not at all obvious) that different minors do not cross.

\subsection*{Offsprings of a minor}
In order to state the first main result, we introduce some terminology and notation.
The convex hull of a subset $A\subset\R^2=\C$ will be denoted by $\ch(A)$.
Let $\ell$ and $\ell_1$ be chords of $\uc$, possibly degenerate, not passing through the center of the disk.
We will write $H(\ell)$ for the smaller open circle arc bounded by the endpoints of $\ell$.
Set $D(\ell)=\ch(H(\ell))$; since $H(\ell)$ is an open arc, $D(\ell)$ does not include $\ell$.
If $\ell_1\in D(\ell)$, then we write $\ell_1<\ell$.
The notation $\ell_1\le\ell$ will mean $\ell_1\in\ol{D(\ell)}$.
Note that, if $\ell_1$ shares just one endpoint with $\ell$ and $\ell_1\le\ell$, then it is not true that $\ell_1<\ell$.
It follows that if $\ell_1\le \ell, \ell_1\ne \ell$ then $|\ell_1|<|\ell|$, where $|\ell|$ denotes the length of $\ell$;
in particular $\ell_1<\ell$ implies $|\ell_1|<|\ell|$.
If $\ell_1<\ell$ (resp., $\ell_1\le \ell$), then we say that $\ell_1$ lies \emph{strictly behind} (resp., \emph{behind}) $\ell$.
Observe that our terminology applies to degenerate chords (i.e., singletons in the unit circle) too;
a degenerate chord $\ell_1=\{b\}$ is strictly behind $\ell$ if and only if $b\in H(\ell)$, and $\ell_1\le \ell$
simply means that $b\in \ol{H(\ell)}$.

Let us now describe an inductive process that shows how \emph{dynamical}
pullbacks of minors of quadratic laminations lead to the construction of the \emph{parametric}
lamination $\qml$. Namely, consider any non-degenerate minor $m\in\qml$.
Suppose that a point $a\in \uc$ lies behind $m$ 
and $\si_2^n(a)$ is an endpoint of $m$ for some minimal $n>0$. Observe
that then $a$ is not periodic as no image of a 
minor is
located behind this minor. Consider \emph{all} numbers $k$ such that
$\si_2^k(a)$ is an endpoint of a minor $m'_k$ with $a<m'_k\le m$ (thus,
$a$ is separated from $m$ by $m'_k$ or $m'_k=m$), and the least such
number $l$. Denote by $m_a$ the pullback of $m'_l$ in $\lam(m'_l)$
containing $a$ such that $\si_2^{l-1}(m_a)$ is a major of $\lam(m'_l)$
and call it an \emph{offspring} of $m$. We also say that
$m_a$ is a \emph{child} of $m'_l$. Observe that periodic minors are
nobody's offsprings. Indeed, if $m'\le m'', m'\ne m''$ are minors,
$\si_2^i(m')=m''$, and $m'$ is periodic, then $\si_2^j(m'')=m'\le m''$
for some $j$, and it is well-known that this is impossible for minors.

\begin{thmA}
Let $m\in\qml$ be a non-degenerate minor. Then
offsprings of a minor $m\in\qml$ are minors too (i.e., they are leaves of $\qml$).
Thus, if a point $a$ lies behind $m$ and is eventually mapped to an endpoint of $m$ under $\si_2$ then
there is a minor $m_a\ni a$ that is eventually mapped to $m$ under $\si_2$.
\end{thmA}

The first claim of Theorem A easily implies the second one.

\subsection*{Renormalization and baby $\qml$s}

The \emph{empty} lamination is the lamination all of whose leaves are degenerate (i.e., are singletons in
$\uc$).

Consider two quadratic laminations $\lam_1$ and $\lam_2$. If
$\lam_2\subset\lam_1$, then we say that $\lam_1$ \emph{tunes} $\lam_2$;
in particular this means that any lamination trivially tunes itself.
If $\lam_2\subsetneqq\lam_1$, then $\lam_1$ is obtained out of $\lam_2$ by
inserting some chords (which become leaves of $\lam_1$) in gaps of $\lam_2$.
If in this setting $\lam_2=\lam(m_2)$ for a non-degenerate periodic minor $m_2$
(we do not exclude the possibility $\lam_2=\lam_1$), then $\lam_1$ is called \emph{renormalizable}.
A lamination $\lam_1$ is \emph{almost non-renormalizable} if there exists no non-empty lamination
$\lam(m_2)\subsetneqq \lam_1$.
We call $\lam_1$ \textbf{almost} non-renormalizable because if it is as above while also $\lam_1=\lam(m_1)$
with non-degenerate periodic minor $m_1$ then, as we saw above, $\lam_1$ is renormalizable, but only in a trivial
way. Observe that in \cite{bot17} almost non-renormalizable laminations are called \emph{oldest ancestors}.

Let $m$ be a non-degenerate periodic minor.
We will write $C(m)$ for the \emph{central set} of $\lam(m)$, i.e., the gap/leaf of $\lam(m)$ containing the center of $\disk$
and, therefore, located between the two majors of $\lam(m)$. Equivalently, $C(m)$ can be called
the \emph{critical} set of $\lam(m)$. 
Then $\si_2(C(m))$ is the convex hull of $\si_2(\uc\cap C(m))$.
This is also a gap or a leaf of $\lam(m)$ having $m$ as a boundary leaf (edge).
We will see that, if $\lam(m_1)$ is renormalizable, then $m_1$ is contained in $\si_2(C(m))$ for some $m$ as above.
Moreover, we can choose $m$ so that $\lam(m)$ is almost non-renormalizable.

All \emph{edges} (i.e., boundary chords) of $\si_2(C(m))$ are leaves of $\qml$.
However, there are also leaves of $\qml$ in $\si_2(C(m))$ that enter the interior of $\si_2(C(m))$.
All these leaves are precisely the minors of all laminations strictly containing $\lam(m)$.
It follows that all renormalizable laminations are represented by minors in gaps of the form $\si_2(C(m))$,
  where $m$ is periodic and such that $\lam(m)$ is almost non-renormalizable.
In other words, 
all minors 
of almost non-renormalizable laminations and all points in $\uc$ form a lamination $\qml^{nr}$
  (``nr'' from \textbf{n}on-\textbf{r}enormalizable)
  whose infinite gaps are a special gap $\ca^c$ and gaps of the form $\si_2(C(m))$, where $m$ is a minor
  such that $\lam(m)$ is almost non-renormalizable.
(There are also finite gaps of $\qml^{nr}$; each such gap is a gap of $\qml$ too,
associated to a non-renormalizable lamination.) Observe that for any
periodic minor $m$ the edges of the 
set $\si_2(C(m))$
are leaves of $\qml$ (they are minors of
laminations that tune $\lam(m)$).
The gap $\ca^c$, the
\emph{combinatorial main cardioid}, is the central gap of $\qml^{nr}$ (and of $\qml$ itself).
By definition, it is bounded by all periodic minors $m$, for which
$\lam(m)$ has an \emph{invariant} finite gap adjacent to $m$,
  or $m$ is an \emph{invariant} leaf of $\lam(m)$.
There are no leaves of $\qml$ in $\ca^c$, except for the edges of $\ca^c$.
The lamination $\qml^{nr}$ was introduced in \cite{bot17}.

Consider a gap $\si_2(C(m))$ of $\qml^{nr}$, where $m$ is a non-degenerate periodic minor
(then $\lam(m)$ is almost non-renormalizable).
Observe that $\si_2(C(m))$ is
invariant under $\si_2^p$, where $p$ is the (minimal) period of $m$.
There is a monotone map $\xi_m$ from the boundary of $\si_2(C(m))$ to $\uc$ that collapses all edges of $\si_2(C(m))$.
We may also arrange that $\xi_m$ semi-conjugates $\si_2^p$ restricted to the boundary of $\si_2(C(m))$ with $\si_2$.
Under $\xi_m$, any leaf $ab\in\qml$ lying in $\si_2(C(m))$ is mapped to a leaf $\xi_m(ab)=\xi_m(a)\xi_m(b)$ of $\qml$.
In this sense, we say that leaves of $\qml$ lying in $\si_2(C(m))$ form a \emph{baby $\qml$}.
Thus, $\qml$ admits the following \emph{self-similar description}:
the lamination $\qml$ is the union of $\qml^{nr}$ and all baby $\qml$s inserted in infinite gaps of the form $\si_2(C(m))$.

To complete this self-similar description we suggest an explicit construction for $\qml^{nr}$ in terms of offsprings.

\begin{thmB}
  The lamination $\qml^{nr}$ is obtained as the set of all offsprings of the edges $m\subset\ca^c$ and the limits of such offsprings.
\end{thmB}

Theorem B parallels the encoding of the Mandelbrot set in terms of
``the Yoccoz combinatorial analytic invariants'' introduced by C.
Petersen and P. Roesch in \cite{pr08}, more specifically see Corollary
3.23 from \cite{pr08} (we are indebted to one the referees for this
remark).

\subsection*{Dynamical generation of the $\qml$}
Theorem B is the basis for a dynamical generation of the $\qml$.
The construction consists of three steps repeated countably many times, and then one
final step. 

\emph{Step 1.}
First, we construct all edges of the combinatorial main cardioid.
The endpoints of these edges can be computed explicitly.

\emph{Step 2.}
For every edge $m$ of the $\ca^c$, we construct all offsprings of $m$.
As follows from Theorem B, taking offsprings is as easy as taking pullbacks of a leaf in an invariant lamination.

\emph{Step 3.}
Take the limits of all offsprings from step 2.
We obtain a lamination behind $m$ with gaps of the form $\si_2(C(m_1))$,
  where $m_1$ is a periodic minor behind $m$ such that $\lam(m_1)$ is almost non-renormalizable.
Drawing these laminations for all edges of $\ca^c$ gives the lamination $\qml^{nr}$.

\emph{Step 4.}
In each gap of the form $\si_2(C(m_1))$ as above, construct chords whose $\xi_{m_1}$-images are leaves constructed at steps 1--3.
In other words, we repeat our construction for each baby $\qml$, and then keep repeating it countably many times.
Let us denote the thus obtained family of leaves of $\qml$ by $\qml^{fr}$.
By \cite{bot17}, $\qml^{fr}$ includes all minors of so-called \emph{finitely
renormalizable} quadratic laminations (``fr'' comes from ``finitely renormalizable'') so that the
only minors that are missing are the ones that correspond
to \emph{infinitely renormalizable} laminations, i.e. laminations
$\lam$ for which there exists a nested infinite sequence of pairwise
distinct laminations $\lam_1\subset \lam_2\subset \dots $ such that
$\lam_n\subset \lam$ for any $n$.

\emph{Step 5.} To get the missing minors we now take the limits of
leaves of $\qml^{fr}$.
Notice that, by \cite{bot17}, these limit minors are, for the most
part, degenerate (i.e., they are singletons in $\uc$). The limit minors that are
non-degenerate are exactly those that correspond to the quadratic laminations $\lam(m)$
that are infinitely renormalizable with the following additional property: $\lam(m)$ coincides with
a q-lamination $\lam_{\sim_m}$ associated to a laminational equivalence $\sim_m$ such that
the corresponding topological polynomial contains a periodic arc
in its topological Julia set.

\subsection*{Possible applications to other parameter slices}
The problem of constructing models of the entire connectedness locus in
degrees greater than $2$ seems to be rather complicated. Even in degree
three there are no known ``global'' models of this space. In this brief
discussion we will, therefore, talk about complex one dimensional
slices of parameter spaces of higher degree polynomials; moreover, for
the sake of simplicity we will only deal with the cubic case. Finally, for
the sake of brevity we assume familiarity with basic concepts of combinatorial
complex dynamics.

One of the main goals of this paper is to develop tools and techniques
that can be used to construct combinatorial models for complex one
dimensional slices of parameter spaces of cubic polynomials. Indeed, by
C. McMullen \cite{Mc00}, slices of the cubic connectedness locus
contain lots of copies of $\M_2$ to which our results apply directly
(in fact, the article \cite{Mc00} contains much more general results).

However otherwise the situation is not as simple. A lot of results show
that in the cubic case various parameter slices are \emph{not} locally
connected. Lavaurs \cite{la89} proved that the cubic connectedness
locus itself is not locally connected. Epstein and Yampolsky
\cite{EY99} showed that the bifurcation locus in the space of real
cubic polynomials is not locally connected either. Buff and Henriksen
\cite{BH01} presented copies of quadratic Julia sets, including not
locally connected Julia sets, in slices of $\M_3$. These are
complications of analytic and topological nature.

There are also combinatorial hurdles that need to be overcome. To begin
with, Thurston's Central Strip Lemma \ref{l:Cstrip} fails already in
the cubic case; e.g., if a cubic lamination admits a critical
quadrilateral $Q$ associated with the critical strip $S$, and a
critical leaf $\ell$, then the forward orbit of $Q$ may come close to
$\ell$ and then enter $S$, a dynamical phenomenon impossible in the
quadratic case because of the Central Strip Lemma. In addition,
Thurston's No Wandering Triangle Theorem (Theorem \ref{t:noWT}) also
fails in the cubic case \cite{bo04, bo08}. This complicates \emph{both}
the task of constructing a combinatorial model of slices of cubic
polynomial spaces \emph{and} the task of applying the idea of the
present paper to such slices even assuming that the laminational model
for (some) slices have been constructed.

In fact, we are not aware of many combinatorial models of such spaces
(even though we believe that a lot of them admit combinatorial models
in terms of laminations). An example one might consider is given in the
paper \cite{bopt16c} which we now discuss. Consider the tripling map
$\si_3:\uc\to \uc$ and fix a critical leaf $D$ of $\si_3$. Moreover,
choose $D$ so that it cannot be a boundary leaf of a periodic Siegel
gap. Then consider the space of all cubic laminational equivalence
relations $\sim$ which have a critical class containing the endpoints
of $D$ (e.g., the endpoints of $D$ may well be a class of this
equivalence relation). Observe that in this case the class containing
the endpoints of $D$ must be finite.

To each such equivalence relation $\sim$ we associate its \emph{minor
set} $m_\sim$ defined as follows. First, if there is a unique critical
set (class) of $\sim$, then $m_\sim$ is the convex hull of its image.
Second, if there are two $\sim$-classes and both are finite, then we
choose the one not containing the endpoints of $D$ and set $m_\sim$ to
be the convex hull of the image of this $\sim$-class.
Finally, consider the remaining case which is as
follows: $\sim$ has a unique periodic \emph{critical} Fatou gap of
period $k$ such that $\si_3^k:U\to \si_3(U)$ is two-to-one. Evidently,
this implies that $\si_3^k:U\to U$ is two-to-one. We show in
\cite{bopt16c} that there is a unique edge $M_\sim$ of $U$ of period
$k$. In this case we set $m_\sim=\si_3(M_\sim)$.

One of the main results of \cite{bopt16c} is that the minor sets
$m_\sim$ can be viewed as tags of their laminational equivalence
relations $\sim_D$ (so that the space of all such laminational
equivalence relations is similar to $\M_2$) while the collection of their
convex hulls will give rise to a lamination $\lam_D$. The corresponding space
of all cubic laminations that admit critical leaf $D$ is $\uc/\sim_D$.
We hope that the ideas and results of this paper can be properly
adjusted to lead to a more explicit description of the structure of
$\uc/\sim_D$ at least for some critical leaves $D$. A likely candidate
for that is the critical leaf $D=\ol{\frac13 \frac23}$, first preimage
of a $\si_3$-fixed angle $0$. This is based upon the fact that if
$D=\ol{\frac13 \frac23}$, then we can prove the Central Strip Lemma for
all laminations admitting $D$, and this allows us to apply similar
arguments to the present paper, in particular concerning pulling back
the minors and thus constructing new minors.

In general, the plan can be as follows. Consider a parameter slice and
assume that its combinatorial model exists. This model will be a
lamination $\lam$ in $\disk$. In order to construct $\lam$, we will
apply a similar procedure to the one described above for $\qml$. Steps
1--3 will be replaced with similar steps. However, step 4 will operate
with genuine baby $\qml$s rather than copies of $\lam$. Thus, the
lamination $\lam$ will consist of a sublamination $\lam^{nr}$ in whose
infinite gaps we insert copies of $\qml$ rather then copies of
$\lam^{nr}$ itself.

Evidently, a lot of details in the actual implementation of the
outlined approach will be very different from what is done in the
current paper.
There are also complications related to the fact
that some quadratic techniques fail for higher-degree polynomials.
Instead of Thurston's technique based on the Central Strip Lemma, we
will have to rely on methods developed in \cite{bopt16c} or, more
generally, in \cite{bopt17}. However, even in the simplest cases of
cubic parameter laminations, a complete implementation of this program
will require at least as much space as this paper. Thus we postpone the
details to future publications. Still, we believe that the sketched
technique should (hopefully!) work for some (but not all) complex one
dimensional slices.

To summarize, we think that while our dynamical approach to the
construction of the Mandelbrot set is quite consistent with the more
static viewpoints of Thurston \cite{thu85}, Keller \cite{kel00},
Lavaurs \cite{lav86, la89}, and Schleicher \cite{sch09}, it is based
upon a familiar pullback construction which has its own advantages, in
particular making it more accessible to those familiar with that
dynamically-based process.

\section{Majors and minors}

In this section, we recall fundamental properties of quadratic laminations.
Since all statements here can be traced back to \cite{thu85},
  we skip references to this
  seminal paper of Thurston until the end of the section (see also \cite{sch09} and \cite{bopt16}
  where some of these results are more 
  fleshed out).
The exposition is adapted to our purposes, and some facts are stated in a different but equivalent form
(see \cite{bmov13} for an extension of this approach to higher degree laminations).
Some proofs are omitted.

\subsection{Notation and terminology}
As usual, $\C$ is the plane of complex numbers identified with the real 2-dimensional vector space $\R^2$.
For any subset $A\subset\C$, we let $\ol A$ denote its closure.
For any set $G\subset\ol\disk$ of the form $G=\ch(G\cap\uc)$, we let $\si_2(G)$ denote the set $\ch(\si_2(G\cap\uc))$.
Chords of $\uc$ on the boundary of $G$ are called \emph{edges} of $G$.
A chord of $\uc$ with endpoints $a$, $b\in\uc$ is denoted by $ab$.
If $a=b$, then the chord is said to be \emph{degenerate}, otherwise it is said to be \emph{non-degenerate}.

We will identify $\R/\Z$ with $\uc$ by means of the map $\theta\in\R/\Z\mapsto \ol\theta=e^{2\pi i\theta}$.
Elements of $\R/\Z$ are called \emph{angles}.
The point $\ol\theta$ will be sometimes referred to as the point in $\uc$ of angle $\theta$.
For example $\ol 0$ and $\ol{\frac 12}$ are the only points of $\uc$ lying on the real axis,
  and $\ol{0\frac 12}$ is the corresponding diameter.
In order to avoid confusion, we will always write $\ol 0$, $\ol{\frac 12}$, $\ol{\frac 14}$ rather than $1$, $-1$, $i$, etc.

Let $M$ be a chord of the unit circle.
We will write $-M$ for the chord obtained from $M$ by a half-turn, i.e., by the involution $z\mapsto -z$.
Let $S$ be the (closed) strip between $M$ and $-M$. Define
the map $\psi:[0, \frac12]\to [0, \frac12]$ by $\psi(x)=2x$ if $0\le x\le \frac14$ and
$\psi(x)=1-2x$ if $\frac14\le x\le \frac12$; the fixed points of $\psi$ are $0$ and $\frac13$. Then
it is easy to see that given a chord $\ell$, we have
$|\si_2(\ell)|=\psi(|\ell|)$. The dynamics of $\psi$ shows that for any non-degenerate chord $\ell$
there exists $n\ge 0$ such that $|\si_2^n(\ell)|\ge \frac13$. Hence if $\pm M$ are the majors of a lamination
then $|M|\ge \frac13$ and $|\si_2(M)|\le \frac13$.
Suppose that $|M|\ge \frac 13$ and that $m=\si_2(M)$ is disjoint from the interior of $S$.
Then the chords $\pm M$ and the strip $S$ are uniquely determined by $m$.
Under the assumptions just made, we call $m$ \emph{minor-like},
set $S=S(m)$, and call it the \emph{central strip} of $m$.
Observe that if $m$ is degenerate, then $S(m)=M=-M$ is a diameter, in particular, it has no interior.
We will write $Q(m)$ for the quadrilateral $\ch(M\cup(-M))$.

\begin{lemma}\label{l:mi-li}
Suppose that $\ell=ab, a\ne b$ is a leaf of a lamination $\lam$ such that $|\ell|\le \frac13$ and $\ol{0}\notin \ol{H(\ell)}$.
Then $\ell$ is minor-like.
In particular if $m$ is a minor and $\ell\le m$, then $\ell$ is minor-like.
\end{lemma}

\begin{proof}
Either two or all four edges of $Q(\ell)$ are leaves of $\lam$. If only
one vertex of $Q(\ell)$ belongs to $H(\ell)$, then at least one edge of
$Q(\ell)$ belongs to $\lam$ and crosses $\ell$, a contradiction.
Hence either two preimages $a'$, $b'$ of points $a$ and $b$,
respectively, belong to $H(\ell)$, or none. Set $\ell'=a'b'$; then
$\si_2(\ell')=\ell$. Suppose that $\ell'\ne \ell$; then
$|\ell'|<|\ell|$. If $\si_2(H(\ell'))$ is $\uc\sm \ol{H(\ell)}$ then
the fact that $|\ell|\le \frac13$ implies that $|\uc\sm
\ol{H(\ell)}|\ge \frac23$ and hence $|H(\ell')|\ge \frac13$, a
contradiction with $|\ell'|<|\ell|\le \frac13$. Hence
$\si_2(H(\ell'))=H(\ell)$ and $|\ell'|\le \frac 16$. Moreover, the
restriction of $\si_2$ to $H(\ell')$ is one-to-one and expanding. It
follows that $\si_2$ has a fixed point in $H(\ell')$. The only fixed
point of $\si_2$ is $\ol 0$, hence we have $\ol 0\in \ol{H(\ell')}$, a
contradiction.
Thus, either $\ell'=\ell$ or $Q(\ell)\cap H(\ell)=\0$ (evidently, all vertices of $Q(\ell)$ cannot belong to $\ol{H(\ell)}$).
In the former case it follows that
$\ell=\ol{\frac13 \frac23}$ is a minor, in the latter case $\ell$ is
minor-like by definition. For the last claim of the lemma, note that if
$m$ is a minor, then $\ol{0}\notin \ol{H(m)}$.
\end{proof}

A \emph{critical chord} is a diameter of $\uc$.
The endpoints of a critical chord are mapped under $\si_2$ to the same point of $\uc$.
A set $G\subset\ol\disk$ of the form $G=\ch(G\cap\uc)$ is said to be \emph{semi-critical} if $G$ contains a critical chord.
Equivalently, a semi-critical set contains the center of the disk.

\subsection{The Central Strip Lemma}
A chord of $\uc$ is said to be \emph{vertical} if it separates $\ol 0$
from $\ol{\frac 12}$, and \emph{horizontal} otherwise. The distinction
between the two types of chords is important for quadratic laminations.

\begin{lemma}
  \label{l:S-forw}
  Let $m$ be a minor-like chord.
  Then $\si_2(S(m)\cap\uc)=\ol{H(m)}$.
\end{lemma}

\begin{proof}
  The set $S(m)\cap\uc$ consists of two arcs, each of length $\le \frac 16$.
Both arcs map to the same arc $A$ of length $\le\frac 13<\frac 12$.
On the other hand, $A$ is bounded by the endpoints of $m$, hence $A=\ol{H(m)}$.
\end{proof}

\begin{lemma}
  \label{l:S-vert}
  Let $m$ be a non-degenerate minor-like chord.
  Then $S(m)$ is bounded by vertical chords.
\end{lemma}

The only degenerate minor-like chord for which the statement fails, is $\ol 0$.

\begin{proof}
  Assume
  that the edges $\pm M$ of $S(m)$ are horizontal.
Then $\ol{0\frac 12}\subset S(m)$,  hence,
$\ol 0\in\ol{H(m)}$ by Lemma \ref{l:S-forw}.
Thus $\ol 0$ belongs to both $S(m)$ and $\ol{D(m)}$.
On the other hand, by definition of minor-like chords, these two sets cannot have common interior points.
It follows that $\ol 0$ is an endpoint of $m$.
Let $\ol\alpha$ be the other endpoint.
Then $\ol{\frac\alpha 2}\in H(m)\cap S(m)$, a contradiction.
\end{proof}

Let us make the following observations.
\begin{enumerate}
  \item[(a)] If $\ell$ is a chord of $\uc$ such that $|\ell|\le \frac 14$, then $|\si_2(\ell)|=2|\ell|$;
  otherwise $|\si_2(\ell)|=1-2|\ell|$.
  \item[(b)] We have $|\si_2(\ell)|> |\ell|$ if and only if $|\ell|< \frac 13$.
  \item[(c)] If $\ell$ is disjoint from the edges $\pm M$ of $S(m)$ and $|\ell|>|M|$, then $\ell$ is a vertical chord in $S(m)$
  (here $m=\si_2(M)$ is minor-like).
  \item[(d)] Any non-degenerate chord eventually maps to a chord of length $\ge \frac13$.
\end{enumerate}

\begin{lemma}[The Central Strip Lemma]
  \label{l:Cstrip}
  Let $m$ be a minor-like chord.
Suppose that the chords $\si_2^n(m)$ do not cross any edge of $S(m)$
for any $n>0$.
\begin{enumerate}
\item If $|\si_2^n(m)|<|m|$ for a minimal $n>0$, then
    $\si_2^{n-1}(m)$ is a vertical chord in $S(m)$ distinct from either edge of $S(m)$, and
    $\si_2^n(m)\le m$;
\item if $\si_2^n(m)\subset S(m)$ for some $n>0$, and $n$ is the smallest positive integer with this property,
  then the chord $\si_2^n(m)$ is vertical.
\end{enumerate}
\end{lemma}

\begin{proof}
We will write $\pm M$ for the edges of $S(m)$. To prove $(1)$, observe
that $|\si_2^{n-1}(m)|>|M|$ which implies that $\si_2^{n-1}(m)\subset
S(m)$ is a vertical chord. Observe now that $(2)$ follows from $(1)$
since if $\si_2^n(m)\subset S(m)$ is horizontal, then $|\si_2^n(m)|\le
|m|/2$.
\end{proof}

\subsection{Minor leaf laminations}
By definition, the Central Strip Lemma, and by observations (a) -- (d), a minor $m$ has the following properties:

\begin{enumerate}
  \item[(SA1)] it is minor-like;
  \item[(SA2)] all $\si_2^n(m)$, where $n\ge 0$, are pairwise
      unlinked and do not cross any edges of $S(m)$;
  \item[(SA3)] for any $n>0$ we have $|\si_2^n(m)|\ge |m|$;
  \item[(SA4)] if $\si_2^n(m)\le m$ for some $n>0$, then
      $\si_2^n(m)=m$ (thus, images of $m$ are disjoint from
      $D(m)\cup S(m)\sm (M\cup -M)$).
\end{enumerate}

For brevity, in what follows we will refer to these properties simply
as SA1, SA2, SA3 and SA4. Clearly, SA3 always implies SA4. Moreover, by
the Central Strip Lemma, if SA1 and SA2 hold for a chord $m$, then SA3
and SA4 for this chord are equivalent.

\begin{definition}[Stand Alone Minor]\label{d:stalone}
A chord $m$ is called a \emph{stand alone minor} if properties SA1--SA3
 hold.
(Then automatically SA4 also holds).
\end{definition}

Note that all points of $\uc$ are stand alone minors. Any stand alone
minor is the minor of a certain quadratic lamination. Any such
lamination can be constructed by ``pulling back'' the minor and all
its images. Such pullbacks are mostly unique but, if $m$ is periodic, allow
for small variations.

In this paper we establish \emph{dynamical conditions} that imply that
certain leaves of a lamination $\lam$ with minor $m$ are minors
themselves. We do this by verifying for them that they are stand alone
minors. This requires checking for them conditions SA1 -- SA3. It turns
out that depending on the location of $\ell$ with respect to $m$ or the
length of $\ell$ with respect to the length of $m$ some of these
conditions easily follow.

\begin{lemma}\label{l:nodrop}
Let $\ell$ be a leaf of a lamination $\lam$ with minor $m$. Then the
following holds.

\begin{enumerate}

\item Choose the least $i\ge 0$ with $|\si_2^i(\ell)|\ge |m|$. Then
    $|\si_2^j(\ell)|\ge |m|$ for any $j\ge i$. Thus, if
    $|\ell|\le |m|$ then
    $|\ell|\le |\si_2(\ell)|\le \dots \le |\si_2^i(\ell)|$
    so that property $\mathrm{SA3}$ holds for
    $\ell$. In particular, $\si_2^t(m)\le m, \si_2^t(m)\ne m$ is impossible.

\item If $m\le \ell$, then no eventual image of $\ell$ crosses the
    edges of $S(\ell)$ so that property $\mathrm{SA2}$ holds for $\ell$.

\end{enumerate}

\end{lemma}

\begin{proof}
(1) By assumption, $|\si_2^i(\ell)|\ge |m|$. If $|\si_2^j(\ell)|<|m|$
for some $j\ge i$, then, by the Central Strip Lemma, for some $k$ the
leaf $\si_2^k(\ell)$ is vertical inside $\ol{S(m)}$, a contradiction
with the vertical pullbacks $\pm M$ of $m$ being the majors of $\lam$.
Observe that $|m|\le \frac13$ as was explained in the paragraph right before Lemma \ref{l:mi-li}.
Hence for each $r, 0\le r\le i-1$ we have $|\si_2^r(\ell)|\le |m|\le \frac13$ which easily
implies that $|\si_2^r(\ell)|\le |\si_2^{r+1}(\ell)|,$ $r=0, \dots, i-1$.

(2) Since the horizontal pullbacks of $\ell$ cross the vertical edges of $S(m)$,
which are leaves of $\lam$, the vertical pullbacks $\pm L$ of
$\ell$ (which are the edges of $S(\ell)$) must be leaves of $\lam$.
Hence eventual images of $\ell$ do not cross an edge of $S(\ell)$, as desired.
\end{proof}

A few well-known results concerning quadratic laminations with a given minor $m$ are summarized
in Theorem~\ref{t:major}; these results can be found in \cite{thu85}, or can be easily deduced from \cite{thu85}.

\begin{theorem}
  \label{t:major}
If $m$ is a stand alone minor, then there exists a quadratic lamination $\lam$ with minor $m$.
Depending on $m$, the following holds. 

\begin{enumerate}
\item If $m$ is non-periodic, then either
\begin{enumerate}
\item
a quadratic lamination $\lam$ with minor $m$ is unique, or

\item if in addition $m$ is non-degenerate, then there are at most two quadratic laminations $\hlam\subset \lam$ with minor $m$
one of which must be a q-lamination $\hlam$ with finite gaps. 

\end{enumerate}

\item If $m$ is periodic and non-degenerate, then there exists a unique q-lamination $\lam$ such that $m$ is its minor.

\item If $m$ is periodic and degenerate, then there are at most four quadratic laminations with $m$ as a minor,
and there exists a unique q-lamination $\hlam$ whose periodic minor $\hat m$ has $m$ as an endpoint. Moreover,
if $m\ne \ol{0}$ then $\hat m$ is non-degenerate.
\end{enumerate}

In any case, there exists a unique q-lamination $\hlam(m)$ such that, if $m$ is not a periodic point, then
any lamination with minor $m$ contains $\hlam(m)$; moreover, if $m$ is non-degenerate and non-periodic, then
all leaves of $\hlam(m)$ are non-isolated in $\hlam(m)$ and all gaps of $\hlam(m)$ are finite.
In case $(1)(b)$, any leaf of $\lam\sm\hlam(m)$ is eventually mapped to vertical edges of $Q(m)$.
\end{theorem}

We can now define a specific lamination $\lam(m)$ with minor $m$.

\begin{definition}\label{d:lamm}
If $m$ is a non-periodic or non-degenerate stand alone minor, define $\lam(m)$ as one of
the laminations from Theorem~\ref{t:major} as follows:
in case (1)(a) the lamination $\lam(m)$ is the unique quadratic lamination with minor $m$;
in case (1)(b), the lamination $\lam(m)$ is the bigger of the two laminations with minor $m$;
in case (2) it is the unique q-lamination with minor $m$.
In 
any case the central set of $\lam(m)$ is denoted by $C(m)$. Finally, the q-lamination $\hlam$ from
Theorem~\ref{t:major} will be denoted by $\hlam(m)$ and will be called the \emph{q-lamination associated with $m$}.
\end{definition}

This defines $\lam(m)$ except for the case when $m$ is a periodic
singleton (which will be done later). By definition, $m$ is the minor
of $\lam(m)$. Observe that the central set $C(m)$ of a lamination
$\lam(m)$ is either a critical leaf (a diameter), a collapsing
quadrilateral, or an infinite periodic quadratic gap.

In the sequel, by a \emph{minor} we mean a stand alone minor or, which
is the same by Theorem \ref{t:major}, the minor of some (not specified)
quadratic lamination. Minors are also identical to leaves of the
$\qml$. The lamination $\lam(m)$ is called the \emph{minor leaf
lamination} associated with a minor $m$. In order to construct
$\lam(m)$, we will describe the process of taking pullbacks of chords.

\begin{definition}[$m$-pullbacks]\label{d:mp}
Let $m$ be a minor-like chord, let $\ell=ab$ be a chord of $\uc$ that
is not linked with $m$. The $m$-\emph{pullbacks} of $\ell$ are defined
as follows. If $\ell=m$, then the $m$-pullbacks are the major(s) $\pm
M$, the edges of $S(m)$. If $\ell\ne m$ is a point in $\uc$, then the
$m$-pullbacks of $\ell$ are points in $\si_2^{-1}(\ell)$. Otherwise, there
are four points in $\si_2^{-1}(\ell\cap\uc)$, and there are two
possible cases. First, $\ell\subset \ol{D(m)}$ in which case all four
points belong to $S(m)$. Then we define the $m$-pullbacks of $\ell$ as
the horizontal pullbacks of $\ell$. Second, $\ell\subset \ol{\disk}\sm
D(m)$ in which case all four points belong to $\ol{\uc\sm S(m)}$. If
$m$ is non-degenerate or $\ell$ is disjoint from $m$, the $m$-pullbacks
of $\ell$ are defined as the two pullbacks of $\ell$ that do not cross
$M$ or $-M$. In the remaining case $m$ is degenerate and is an endpoint
of $\ell$; then we define the $m$-pullbacks of $\ell$ to be the
pullbacks of $\ell$ that have length $\le \frac14$.
\end{definition}

In the last case in Definition~\ref{d:mp}, if $m\ne \ol{0}$ or if
$m=\ol{0}$ but $\ell\ne \ol{0\frac12}$, there are exactly two
$m$-pullbacks of $\ell$ while if $m=\ol{0}$ (hence $M=\ol{0 \frac 12}$) and
$\ell=M$ then there are four such pullbacks: $\ol{0\frac14},$
$\ol{\frac14\frac12},$ $\ol{\frac12 \frac34}$ and $\ol{\frac34 0}$.

Observe that if (degenerate) $m\ne \ol{0}$ is an endpoint of $\ell$ then the $m$-pullbacks of $\ell$ are horizontal.
Indeed, in that case $M\ne \ol{0\frac12}$ is a diameter of $\disk$ with endpoints $\pm a$.
We may assume that $a$ is in the upper half-plane.
Then $m=\si_2(a)<M$.
If $\ell$ is small, then the $m$-pullbacks of $\ell$ are two short chords coming out of the points $\pm a$
(the other two candidate pullbacks are of length $>\frac14$).
Clearly, both chords are horizontal.
As we continuously increase the length of $\ell$, its pullbacks also continuously increase.
The longest option for $\ell$ is still shorter than a half-circle, hence these chords are $m$-pullbacks of $\ell$.
If at some moment they stop being horizontal, then at this moment the endpoints of these chords not in $M$ must become either $\ol 0$ or $\ol{\frac12}$.
Hence their common image $\ell$ must have an endpoint $\si_2(\ol{0})=\ol{0}$.
However $\ell$ cannot have $\ol{0}$ as an endpoint, a contradiction.

Importantly, there is no way of making $m$-pullbacks depend continuously on $m$.
This is why the definition of $m$-pullbacks may not look very natural.
Observe the following.
If $m$ is a minor, then any chord of the form $\si_2^n(m)$ is an $m$-pullback of $\si_2^{n+1}(m)$ for $n\ge 0$.
Indeed, this statement is non-trivial only for non-degenerate $m$.
In this case $m$-pullbacks are determined by the property that they do not cross the edges of $S(m)$
(by property (4) of minors, iterated images of $m$ never enter $S(m)\sm (M\cup -M)$).
The following theorem complements Theorem \ref{t:major}; recall that in case when
$m$ is non-degenerate, or degenerate and non-periodic, $\lam(m)$ was defined above
(see Theorem~\ref{t:major}).

\begin{theorem}
  \label{t:mpb}
  If $m$ is a non-degenerate or non-periodic minor
  then iterated $m$-pullbacks of iterated $\si_2$-images
  of $m$ are dense in $\lam(m)$.
\end{theorem}

In fact, Theorem \ref{t:mpb} inspires the \emph{definition} of $\lam(m)$ in the only remaining case
when $m$ is a periodic singleton; in that case we define $\lam(m)$ as the closure of the family of all
iterated $m$-pullbacks of $M$ where $M$ is the diameter mapped to $m$ by $\si_2$.

\subsection{Classification of dynamic gaps}
The key tool that allowed Thurston to succeed in establishing a
complete classification of gaps of quadratic laminations was Theorem
\ref{t:noWT} (No Wandering Triangles Theorem). Let $\Delta$ be a
triangle with vertices in $\uc$. It is said to be \emph{wandering} if
all $\si_2^n(\Delta)$ have non-empty disjoint interiors for $n\ge 0$.

\begin{theorem}
\label{t:noWT}
Wandering triangles do not exist.
\end{theorem}

The first step in the classification of all gaps is the following
corollary.

\begin{cor}
  \label{t:finOcri}
Let $G$ be a gap of a quadratic lamination $\lam$. Then an eventual image of $G$
either contains a diameter or is periodic and finite.
\end{cor}

Semi-critical gaps are classified as follows:
\begin{itemize}
  \item strictly preperiodic critical finite gaps with more than 4 edges;
  \item \emph{collapsing quadrilaterals}, i.e., quadrilaterals that are mapped to non-degenerate leaves;
  \item \emph{collapsing triangles}, i.e., triangles with a critical edge;
  \item \emph{caterpillar gaps}, i.e., periodic gaps with a critical edge.
  \item \emph{Siegel gaps}, i.e., infinite periodic gaps $G$ such that $G\cap \uc$ is a Cantor set, $\si_2^n$ maps $G$
      onto itself for some $n$, and $\si_2^n$ restricted to the boundary of $G$ is semi-conjugate to an irrational rotation of the circle under the map that collapses all edges of $G$ to points.
\end{itemize}
All edges of a caterpillar gap are eventually mapped to the critical edge.
Any caterpillar gap has countably many edges and countably many vertices.

Let $A\subset\uc$ be a compact set. Denote by $\si_d:\uc\to\uc$ the
$d$-tupling map that takes $z$ to $z^d$ for any $d\ge 2$. We say that
$\si_d:A\to\si_d(A)$ has \emph{degree $k$ covering property} if there
is a degree $k$ orientation preserving covering $f:\uc\to\uc$ such that
$\si_d|_A=f|_A$ and such $k$ is minimal.

\begin{prop}
  \label{p:cov-prop}
Consider a gap $G$ of a quadratic lamination $\lam$ such that no edge of $G$ is a critical leaf.
Then the map $\si_2:G\cap\uc\to\si_2(G\cap\uc)$ has degree $k$ covering property, where $k=1$ or $2$.
\end{prop}

A bijection from a finite subset $A$ of $\uc$ to itself is a \emph{combinatorial rotation} if it preserves the cyclic order of points.
Thus, a combinatorial rotation $f:A\to A$ is a map
which extends to an orientation preserving homeomorphism $g:\uc\to\uc$, topologically conjugate to a Euclidean rotation.
A gap $G$ of a quadratic lamination $\lam$ is \emph{periodic} if
$\si_2^p(G)=G$ for some $p>0$;
the smallest such $p$ is the \emph{period} of $G$.
If $G$ is of period $p$, then
$\si_2^p$ restricted to $G\cap\uc$ is the \emph{first return map} of $G$.
By Proposition \ref{p:cov-prop} the first return map of a finite periodic gap is a combinatorial rotation.
Moreover, if $\lam$ has no critical leaves then the first return map of an infinite periodic gap $G$
has the degree 2 covering property and $G\cap \uc$ is a Cantor set.

\begin{lemma}
  \label{l:1stret}
  Let $G$ be a periodic gap of a quadratic lamination $\lam$, and $f:G\cap\uc\to G\cap\uc$ its first return map.
\begin{enumerate}
  \item If $G$ is finite, then $f$ is a transitive combinatorial rotation.
  In particular, for any $a$, $b\in G\cap\uc$ such that $ab$ is not an edge of $G$, the chord $f^k(ab)$ crosses $ab$ for some $k>0$.
  \item If $a\ne b\in \uc\cap G$ and neither $a$ nor $b$ eventually maps to $\ol 0$,
   then $f^k(ab)$ is vertical for some $k\ge 0$.
   This is true, e.g., if the interior of $G$ contains the center of $\disk$, the lamination $\lam$ is non-empty,
   and $a$, $b$ are arbitrary points in $G\cap\uc$.
\end{enumerate}
\end{lemma}

\begin{proof}
The only claim that is not explicitly contained in \cite{thu85} is the last one.
Assume, by way of contradiction, that $f^k(ab)$ is horizontal for all $k\ge 0$.
Then, for every $k$, either both $f^k(a)$ and $f^k(b)$ are in the open upper half of $\disk$,
 or both in the open lower half of $\disk$.
Suppose that a point $x\in\uc$ never maps to $\ol 0$.
Define the \emph{address} of $x$ as $U$ if $x$ is above $\ol{0\frac 12}$ and as $L$ otherwise.
(The symbols $U$ and $L$ come from ``Upper'' and ``Lower'').
The \emph{itinerary} of $x$ is an infinite word in the alphabet $\{U,L\}$ consisting of addresses
of all $f^k(x)$ for $k\ge 0$. Similarly, we can define \emph{finite
itineraries} of length $N$ if, instead of all $k\ge 0$, we take all $k$
such that $0\le  k< N$.
It is easy to see that the locus of points with a given finite itinerary is 
an arc in $\uc$.
Moreover, this arc has length $2^{-N}$, where $N$ is the length of the
itinerary. It follows that every infinite itinerary defines at most one
point. In particular, since by the assumption $a$ and $b$ have the same itinerary, we
conclude that $a=b$, a contradiction.

If $G$ contains the center of $\disk$ in its interior, then $\lam$ does not have critical leaves.
Hence $\si_2$ has a degree $k$ covering property on $G$, with $k=1$ or $2$.
We claim that then $\ol{0}\notin G$.
Indeed, suppose otherwise. Then it is easy to see that $G$ is invariant and, hence, $f=\si_2$. Now,
if $G$ is finite, then $f$ fixes $\ol 0$, hence cannot act as a transitive combinatorial rotation.
If $G$ is infinite, then the fact that $\lam$ has no
critical leaves implies that $f$ has degree 2 covering property on $G$.
Using the density of $\bigcup_{n\ge 0} \si_2^{-n}(\ol 0)$ in both $\uc$ and $G\cap \uc$,
 we conclude that $G=\ol{\disk}$ and $\lam$ is the empty lamination, a contradiction.
 Hence we may assume that $\ol{0}$
(and, therefore $\ol{\frac12}$) do not belong to $G$.
Since $G$ is
periodic, the points $\ol{0}$ and $\ol{\frac12}$ do not belong to iterated $\si_2$-images of $G$ either.
This implies that if $a\ne b\in G\cap\uc$ then,
by the first paragraph, $f^k(ab)$ is vertical for some $k\ge 0$.
\end{proof}

\subsection{Classification of parameter gaps}
\label{ss:qml-gaps}
Thurston classified all gaps of $\qml$ (see Theorem II.6.11 of \cite{thu85}); we outline this classification below.

Suppose first that $G$ is a finite gap of $\qml$.
Then $G$ is strictly preperiodic under $\si_2$.
Moreover, it is the $\si_2$-image of a finite central set $C$ in a quadratic lamination $\lam$.
The gap $C$ has 6 edges or more.
Conversely, if a quadratic lamination $\lam$ has a finite central gap $C$ with 6 or more edges,
  then $\si_2(C)$ is a finite gap of $\qml$.
To summarize, finite gaps of $\qml$ are precisely finite gaps of quadratic q-laminations that are the images of their central gaps.

Suppose now that $G$ is an infinite gap of $\qml$.
Then all edges of $G$ are periodic minors.
It may be that $G=\ca^c$.
Otherwise, there is a unique edge $m_G=m$ of $G$ such that all $\ell\le m$ for any other edge $\ell$ of $G$.
Then $G\subset\si_2(C(m))$.
However, only the edge $m$ is on the boundary of $\si_2(C(m))$.
Other edges of $G$ enter the interior of $\si_2(C(m))$.

It is useful to think about $G$ as a copy of $\ca^c$ inserted into $\si_2(C(m))$.
To make this more precise, observe that there is a monotone continuous map $\xi_m:\uc\to\uc$ with the following properties.
Every complementary component of $\si_2(C(m))$ in $\uc$, together with endpoints of the edge of $\si_2(C(m))$ that bounds it,
 is mapped to one point.
The map $\xi_m$ semi-conjugates the restriction $\si_2^p|_{\si_2(C(m))\cap\uc}$ with $\si_2$.
Here $p$ is the period of $\si_2(C(m))$.
The map $\xi_m$ is almost one-to-one on $\si_2(C(m))\cap\uc$ except that it identifies the endpoints of every edge of $\si_2(C(m))$.
There is a unique map $\xi_m$ with the properties just listed.
Then $G$ is a copy of $\ca^c$ in the sense that the $\xi_m$-images of the edges of $G$ are precisely the edges of $\ca^c$.
Moreover, $\xi_m$-pullbacks are well defined for all edges of $\ca^c$.
Indeed, no endpoint of an edge of $\si_2(C(m))$ has period $>1$ under the first return map to $\si_2(C(m))$.
Note that, as a consequence, the period of $m$ is the smallest among the periods of all edges of $G$.
Other periods are integer multiples of the period of $m$.

The case of $\ca^c$ is somewhat special as this gap is not associated with any minor.
Thurston suggested to think of $\ca^c$ as being associated with the degenerate minor $\ol 0$.
Indeed, with these understanding, most properties of infinite gaps of $\qml$ extend to the case of $\ca^c$.

\section{Derived minors, children, and offsprings: proof of Theorem A}


Let us begin with a technical lemma.

\begin{lemma}\label{l:disj}
Let $\ell$ be a leaf of a quadratic lamination $\lam$ where either $\lam=\lam(m)$, and $m$ is
not a periodic point, or $\lam$ is a q-lamination.
Moreover, let $\si_2^i(\ell)\cap \si_2^j(\ell)\ne
\0$ for some $0\le i<j$. Then $\si_2^i(\ell)$ is a periodic leaf. In
particular, if $\ell<m$ and $\si_2^n(\ell)=m$ for some $n$, then all
leaves $\si_2^i(\ell)$ with $\si_2^i(\ell)\le m, \si_2^i(\ell)\ne m,$
are pairwise disjoint.
\end{lemma}

\begin{proof} By 
Definition~\ref{d:lamm}, the lamination $\lam(m)$ is either a q-lamination, or a tuning of a q-lamination with finite gaps.
Thus, either $\si_2^i(\ell)=\si_2^j(\ell)$ is a periodic
leaf (mapped to itself under $\si_2^{|j-i|}$), or both leaves $\si_2^i(\ell)$ and $\si_2^j(\ell)$
are contained in the same
finite periodic gap $G$ of some q-lamination.
However, in the latter case, neither leaf in question can be a diagonal of $G$ because,
 by Lemma~\ref{l:1stret}, eventual images of such diagonals cross each other.
Thus again $\si_2^i(\ell)$ is a periodic leaf.

Now, let $\ell<m$, set $n$ to be the smallest number such that $\si_2^n(\ell)=m$, and assume that
 $\si_2^i(\ell)\cap \si_2^j(\ell)\ne \0$ for some $0\le i<j\le n$.
Then, by the above, $\si_2^i(\ell)$ and $m$ belong to the same periodic orbit of leaves.
However, $\si_2^r(m)\le m$ is impossible unless $\si_2^r(m)=m$, by Lemma~\ref{l:nodrop} .
\end{proof}

Let us now describe several ways of producing new minors $\ell$ from
old ones, cf. part (a) of Lemma II.6.10a in \cite{thu85}.
We say that a leaf $\ell$ \emph{separates} the leaf $\ell'$ from the leaf
$\ell''$ if $\ell'$ and $\ell''$ are contained in distinct components
of $\disk\sm \ell$ (except, possibly, for endpoints). In particular,
this means that $\ell\ne \ell'$ and $\ell\ne \ell''$.

\begin{definition}[Derived minors and children]\label{d:derived-min}
Let $m$ be a minor. Let $m_1\le m$ be a leaf of $\lam(m)$ such that
eventual images of $m_1$ do not separate $m_1$ from $m$ and never equal
a horizontal edge of the critical quadrilateral $Q(m)$.
Then $m_1$ is called a \emph{(from $m$) derived minor}. If, in
addition, $m_1$ is mapped onto $m$ under a suitable iterate of $\si_2$,
then $m_1$ is called a \emph{child} of $m$.
\end{definition}

By Proposition~\ref{p:der-min} proved below, every derived minor is a minor, justifying its name.
If the central gap $C(m)$ of $\lam(m)$ is distinct
from $Q(m)=\ch(M\cup(-M))$ where $M$ is a major of $\lam(m)$ (i.e., if
the horizontal edges of $Q(m)$ are not leaves of $\lam(m)$), then
automatically no image of $m_1$ equals a horizontal edge of the
critical quadrilateral $Q(m)$.
Observe, that if $\ell\le m$ and $n$ is the minimal number such that $\si_2^n(\ell)=m$, then to verify that $\ell$
is a from $m$ derived minor it suffices to verify that $\ell$ never maps to
a horizontal edge of $Q(m)$ and that $\si_2^i(\ell)$ does not separate
$\ell$ from $m$ for $0<i<n$ (for $i\ge n$ this will hold automatically by
Lemma~\ref{l:nodrop}).

\begin{prop}
\label{p:der-min} Let $m$ be a non-degenerate minor. If a leaf $m_1\in
\lam(m)$ is a from $m$ derived minor, then $m_1$ is a minor. Moreover,
the horizontal edges of the collapsing quadrilateral $Q(m_1)$ belong to
$\lam(m)$, and if $\si_2^n(m_1)=m$ is the first time $m_1$ maps to $m$,
then $\si_2^{n-1}(m_1)$ is a major of $\lam(m)$.
\end{prop}

\begin{proof}
By Lemma~\ref{l:mi-li}, the chord $m_1$ is minor-like, i.e., SA1 holds.
Let us now check SA2.
Since $m_1$ is a leaf of $\lam$, the chords $\si_2^k(m_1)$
are unlinked for $k\ge 0$. By way of contradiction, suppose that for
some $k\ge 0$ the chord $\si_2^k(m_1)$ crosses an edge $M_1$ of
$S(m_1)$. Then it crosses the edge $-M_1$ since otherwise
$\si_2^{k+1}(m_1)$ would cross $m_1$. On the other hand, we know that
$\si_2^k(m_1)$ cannot cross edges of $S(m)$, hence $\si_2^k(m_1)\subset S(m)$.
Since $\si_2^k(m_1)$ is a leaf of $\lam(m)$, it cannot be vertical.
Thus $\si_2^k(m_1)$ is horizontal and separates the two horizontal edges of $Q(m_1)$.
However, this implies that $\si_2^{k+1}(m_1)$ separates $m_1$ from $m$.
A contradiction with the assumption that $m_1$ is a derived minor.

Property SA3 follows from Lemma~\ref{l:nodrop}.
To prove the next to the last claim, observe that $m_1$ must have two pullbacks in $\lam(m)$,
and its vertical pullbacks cannot be leaves of $\lam(m)$ as they are
longer than the majors of $\lam(m)$. The last claim follows from the
definition of a derived minor.
\end{proof}

Next we prove a simple but useful technical lemma.

\begin{lemma}\label{l:techno} The following facts hold.

\begin{enumerate}

\item If $m$ is a minor, $\ell$ is a chord such that
    $\si_2^k(\ell)=m$ with $k$ minimal, and
    $\si_2^i(\ell)$ is a horizontal edge of $Q(m)$, then $i=k-1$.

\item If $m'$ and $m''$ are two distinct non-disjoint minors, then
    they are edges of the same finite gap $G$ of $\qml$. The gap
    $G$ is the image of a finite critical gap of some q-lamination
    and is pre-periodic so that the forward orbit of $m'$ does not
    contain $m''$, and the forward orbit of $m''$ does not contain
    $m'$. Thus, if $m_1\le m$ are two minors and $m$ is an eventual
    image of $m_1$, then $m_1<m$.

\end{enumerate}

\end{lemma}

\begin{proof}
(1) By the choice of $k$, we have $i\ge k-1$.
Also, $\si_2^k(\ell)=m$ implies that $\si_2^{k-1}(\ell)$ is a horizontal edge of $Q(m)$. If, for some
$i>k-1$, the leaf $\si_2^i(\ell)$ is a horizontal edge of $Q(m)$, then $m$ is a
periodic minor whose orbit includes a horizontal edge of $Q(m)$.
However, the orbit of a periodic minor $m$ includes a major of
$\lam(m)$ but does not include horizontal edges of $Q(m)$.

(2) Easily follows from the No Wandering Triangles Theorem.
\end{proof}

The next lemma is based on Proposition~\ref{p:der-min}.

\begin{lemma}\label{l:find-min}
Let $m$ be a minor.
Let $a\in H(m)$ be a point and $n$ be the smallest integer such that $\si_2^n(a)$ is an endpoint of $m$.
Let $\ell$ be a leaf of $\lam(m)$ with endpoint $a$ chosen so that $\si_2^{n-1}(\ell)$ is a major of $\lam(m)$.
Among all iterated images of $\ell$ that separate $a$ from $m$, choose the one closest to $m$; call it $\ell'$.
If no iterated image of $\ell$ separates $a$ from $m$, set $\ell'=\ell$.
Then $\ell'$ is a from $m$ derived minor.
\end{lemma}

The leaf $\ell'$ is well defined as there are only finitely many iterated images $\ell''\le m$ of $\ell$
(this is because no iterated image of $m$ is behind $m$, which follows from the Central Strip Lemma).
Observe that $\ell$ defined in the lemma never maps to a horizontal edge of $Q(m)$
because $\si_2^{n-1}(\ell)$ is a major of $\lam(m)$,
and majors of $\lam(m)$ do not map to horizontal edges of $Q(m)$.

\begin{proof}
By the choice of $\ell$ the leaf $\ell'$ is a pullback of $m$ in
$\lam(m)$ such that no forward image of $\ell'$ separates $m$ from
$\ell'$ and no image of $\ell'$ is a horizontal edge of $Q(m)$. Hence
by definition $\ell'$ is a from $m$ derived minor.
\end{proof}

We are ready to prove Theorem A. Observe that by Theorem A a minor
$\tilde m<m$ is an offspring of a minor $m$ iff $\si_2^n(\tilde m)=m$
for some $n>0$.

\begin{proof}[Proof of Theorem A] Let $m$ be a minor. Let $a\in H(m)$ be a
point and $n$ be a minimal integer such that $\si_2^n(a)$ is an
endpoint of $m$. Let us find the leaf $\ell'$ as in
Lemma~\ref{l:find-min}. Then $\ell'\in \lam(m)$ is a from $m$ derived
minor which is a child of $m$. If $a$ is an endpoint of $\ell'$, we are
done. Otherwise we apply Lemma~\ref{l:find-min} to $a$ and $\ell'$.
Observe that this time we will find the appropriate pullback of $\ell'$
with endpoint $a$ in the lamination $\lam(\ell')$, not in $\lam(m)$,
and our choice will be made to make sure that this pullback of $\ell'$
does not pass through a horizontal edge of $Q(\ell')$. On the other
hand, the pullback of $\ell'$ that we will find does eventually map to $m$.
After finitely many steps the just described process will end,
 and we will find the desired offspring of $m$ with endpoint $a$.
\end{proof}

We complete this section with two lemmas that will be used later on.

\begin{lemma}\label{l:new-minor-1}
Let $m$ be the minor of a lamination $\lam$.
Then any leaf $\ell\in\lam$ such that $\ell\le m$ and $|\ell|>\frac{|m|}2$ is a minor.
In particular, if $\ell\le m$ is sufficiently close to $m$, then $\ell$ is a minor.
\end{lemma}

\begin{proof}
By Lemma~\ref{l:mi-li}, the chord $\ell$ is minor-like so that SA1 holds for $m$.
Let us verify property SA2 for $\ell$. Let $|m|=2\lambda$. Then the
width of the strip $S(m)$ is $\lambda$. If $\ell\in \lam,$ $\ell\le m$
and $|\ell|>\frac{|m|}2=\lambda$, then, by Lemma \ref{l:nodrop}(1), we have
$|\si^i_2(\ell)|>\lambda$ for every $i>0$. Hence eventual images of
$\ell$ do not enter the interior of $S(m)$ horizontally. On the other
hand, they cannot enter the interior of $S(m)$ vertically since the
edges $\pm M$ of $S(m)$ are the majors of $\lam$. Since $\ell\in \lam$,
eventual images of $\ell$ do not cross the majors $\pm M$ of $\lam$.
Hence they do not intersect $S(\ell)$ at all, and $\ell$ has property SA2.
By Lemma~\ref{l:nodrop}, the leaf $\ell$ also has property SA3.
Hence $\ell$ is a stand alone minor.
\end{proof}

Lemma~\ref{l:between} describes other cases when a minor can be
discovered; assumptions of Lemma~\ref{l:between} reverse those of
Proposition~\ref{p:der-min}.

\begin{lemma}\label{l:between}
Let $m$ be the minor of a lamination $\lam$ and $\ell\in \lam$ is a minor-like leaf such that $m\le \ell$.
Moreover, suppose that $m\le \si_2^n(\ell)\le \ell$ is false for any $n>0$.
Then $\ell$ is a minor.
In particular, this is the case if $m\le \ell\le \widehat m$ where $\widehat m\in \lam$ is a minor, $\si_2^n(\ell)=\widehat m$
 for some $n$, and no leaf $\si_2^i(\ell)$, $0<i<n$, separates $m$ from $\ell$.
\end{lemma}

\begin{proof}
By the assumptions, SA1 holds for $\ell$.
By Lemma~\ref{l:nodrop}(2), property SA2 also holds for $\ell$.
To verify SA3, assume, by way of contradiction,
that for some minimal $n>0$ we have $|\si_2^n(\ell)|<|\ell|$. Then by
the Central Strip Lemma (which applies because of SA2), the leaf
$\si_2^{n-1}(\ell)\subset S(\ell)$ is vertical. The fact that $m$ is
the minor of $\lam$ now implies that $\si_2^{n-1}(\ell)$ must be a
vertical leaf in $\ol{S(\ell)\sm S(m)}$ which in turn implies that
$m\le \si_2^n(\ell)\le \ell$, a contradiction.
Thus, SA3 holds for $\ell$, and $\ell$ is a minor.

To prove the second claim of the lemma notice that by the Central Strip
Lemma, no eventual image of $\widehat m$ is behind $\widehat m$. Together with the
assumptions of the lemma on $\ell$ it implies that no eventual image of
$\ell$ separates $\ell$ from $\widehat m$.
By the above, $\ell$ is a minor.
\end{proof}

\section{Coexistence and tuning}
We start with a general property of minor leaf laminations.
A chord $\ell$ is said to \emph{coexist} with a lamination $\lam$ if no leaf of $\lam$ is linked with $\ell$.

\begin{lemma}
  \label{l:vert-vs-hor}
  Let $m$ be a minor, and $\lam(m)$ the corresponding minor leaf lamination.
If $Q\subset S(m)$ is a collapsing quadrilateral whose vertical edges coexist with $\lam(m)$,
  then $Q$ is contained in the critical gap of $\lam(m)$.
\end{lemma}

\begin{proof}
If a horizonal edge $\ell_h$ of $Q$ and a leaf $\ell\in\lam(m)$ cross
in $\disk$, then, since $\ell$ cannot cross the vertical edges of $Q$,
$\ell$ must cross $-\ell_h$. Thus, $\ell$ is a vertical leaf of
$\lam(m)$ in $S(m)$, a contradiction. Hence horizonal edges of $Q$ also
coexist with $\lam(m)$. Since $m$ is non-degenerate, $\lam(m)$ has no
critical leaves. Thus $Q$ is contained in the critical gap of
$\lam(m)$.
\end{proof}

Coexistence of chords turns out to be stable under $\si_2$.


\begin{lemma}
  \label{l:com-stab}
  Suppose that a chord $\ell$ coexists with a quadratic lamination $\lam$.
  Then $\si_2(\ell)$ also coexists with $\lam$.
\end{lemma}

\begin{proof}
Assume the contrary: $\si_2(\ell)$ is linked with some leaf $ab$ of $\lam$.
The chords $\pm\ell$ divide the circle $\uc$ into four arcs, which will be called the $\pm\ell$-arcs.
The two $\si_2$-preimages of $a$ are in the opposite (=not adjacent) $\pm\ell$-arcs.
Similarly, the two preimages of $b$ are in the remaining opposite $\pm\ell$-arcs.
It follows that any pullback of $ab$ in $\lam$ 
crosses $\ell$ or $-\ell$, a contradiction.
\end{proof}

Two laminations $\lam_1$, $\lam_2$ are said to \emph{coexist} if no
leaves $\ell_1\in \lam_1$ and $\ell_2\in \lam_2$ cross. Thus,
coexistence of quadratic laminations is a symmetric relation.

\begin{lemma}
  \label{l:tune-m1-le-m}
Let $m_1$ be a 
minor that is an offspring of a 
non-degenerate minor $m_0$. If $\lam(m_1)$ coexists with some
quadratic lamination $\lam\ne \lam(m_1)$ with minor $m$, then either
$m$ is an endpoint of $m_1$ and $\lam$ is the corresponding lamination
with a critical leaf, or $\lam$ is the q-lamination associated to
$\lam(m_1)$, or $m_1<m_0\le m$.
\end{lemma}

\begin{proof}
We assume from the very beginning that $m$ is not an endpoint of $m_1$.
It is easy to see that $m_1$ is non-periodic since $m_1$ is an
offspring of $m_0$. Hence by Theorem \ref{t:major} the lamination
$\lam(m_1)$ contains the critical quadrilateral $Q(m_1)$, and
$\lam(m_1)$ is obtained from the q-lamination $\hlam(m_1)$ (with finite
gaps and all leaves being non-isolated) by inserting vertical edges of
$Q(m_1)$ in its central gap $C(\hlam(m_1))$ (in this way one adds
$Q(m_1)$ to $\hlam(m_1)$) and then pulling them back within
$\hlam(m_1)$. The only two laminations that tune $\lam(m_1)$ are the
ones whose minors are endpoints of $m_1$. Hence, by our assumption, it
follows that $\lam$ cannot have any leaves that do not belong to
$\lam(m_1)$. In other words, $\lam\subsetneqq \lam(m_1)$. Since the
majors $\pm M$ of $\lam$ are leaves of $\lam(m_1)$, then they are
located so that $S(m)\supset S(m_1)$ and hence $m_1\le m$.

Consider the case when $m_1\in \lam$. If the majors $\pm M_1$ belong to
$\lam$, it follows that $m=m_1$. Since $m=m_1$ is not periodic, the
central gap of $\lam$ must be finite. Since by Theorem \ref{t:major}
the horizontal edges of $Q(m_1)$ are limits of leaves of $\hlam(m_1)\subset \lam(m_1)$,
$Q(m_1)$ must be a gap of $\lam$, and it follows that $\lam=\lam(m_1)$. Suppose that
$m_1\in \lam$ but $\pm M_1$ do not belong to $\lam$. Let $\widehat C$
be the critical set of the q-lamination $\hlam$ associated to $m_1$.
Then the horizontal edges of $Q(m_1)$ must be edges of $\widehat{C}$ and leaves of $\lam$.
Indeed, some leaves of $\lam$ must map to $m_1$, and the vertical edges of $Q(m_1)$ do not belong to $\lam$.
Hence the critical set of $\lam$ is a gap $H$ containing the horizontal edges of $Q(m_1)$ in the boundary.
Consider two cases.

If $H$ is finite, then the fact that $\pm M_1$ do not
belong to $\lam$ and the fact that edges of $\widehat C$ are approached
from the outside of $\widehat C$ by leaves of $\hlam\subset \lam(m_1)$
imply that $H\subset \widehat C$ is different from $Q(m_1)$.
Since no edges of $H$ can cross $\pm M_1$ and the images of the edges of $H$ must be edges of
$\si_2(\widehat C)$ (otherwise some of their eventual images will cross), we have
$H=\widehat C$ and, hence, $\lam=\hlam$ is the q-lamination associated to $m_1$.

If $H$ is infinite, then $H$ is a quadratic Fatou gap, and $m_1$ is an
edge of its image; it is well known that then $H$ is periodic of
period, say, $n$. It is known that there is a unique periodic edge $M$
of $H$, and it is of period $n$. Moreover, $M$ and its sibling $-M$ are
the majors of the unique lamination that has $H$ as its gap; this
lamination is in fact a q-lamination and, evidently, it has to coincide
with $\lam$ so that $m=\si_2(M)$ is an edge (actually, unique periodic
edge) of $\si_2(H)$. It is known that all edges of $H$ eventually map
to $m$ (it is a consequence of the Central Strip Lemma), in particular
so does $m_1$ (which is an edge of $\si_2(H)$) and $m_0$ (which is an
eventual image of $m_1$).

The Central Strip Lemma also imposes restrictions on possible locations of iterated images of $H$.
Namely, the entire gap $\si_2(H)$ is located under $m$
while all other iterated images of $H$ are located on the other side of $m$.
Now, $m_0$ is a minor of some lamination and an eventual image of $m_1$.
Since $m$ is an eventual image of $m_1$, it follows that $m$ is an eventual image of $m_0$.
If $m_0$ is an edge of some iterated image of $H$ different from $\si_2(H)$, then
 $m_1\le m_0$ implies $m\le m_0$ (recall that both $m_1$ and $m$ are edges of $\si_2(H)$).
Since $m_0$ is eventually mapped to $m\le m_0$, we must have $m=m_0$, and we are done in this case.
Thus we may assume that $m_0$ is an edge of $\si_2(H)$.
Since the only edge $\ell$ of $H$ such that $m_1<\ell$ is the edge $m$, the fact that $m_1<m_0$ implies again that $m_0=m$.
All that covers the ``trivial'' cases included in the theorem.

Now, if $m_1\notin \lam$, then $m_1$ is a diagonal of a gap $G$ of $\lam$ whose edges are leaves of $\lam(m_1)$.
Since $m_1$ is approached by uncountably many leaves of $\lam(m_1)$ from at least one side,
$G\cap \uc$ is infinite and uncountable (in particular, $G$ is not an iterated pullback of a caterpillar gap).
Also, $G$ is not an iterated pullback of a periodic Siegel gap as otherwise $m_1$, being a diagonal of $G$, will
have some eventual images that cross. Since $G$ is infinite, it is
eventually precritical and an image $\si_2^i(G)=H$ of $G$ is a
periodic critical quadratic Fatou gap containing as a diagonal the leaf $\si_2^i(m_1)$.
As in the previous paragraph, there is a unique periodic edge $M$ of $H$, and it is of period $n$. Moreover, $M$ and
its sibling $-M$ are the majors of a unique lamination that has $H$ as its gap;
 this lamination is in fact a q-lamination and, evidently, it coincides with $\lam$ so that $m=\si_2(M)$.

The majors $\pm M_1$ coexist with $\lam$ and cannot cross edges
of $H$. Hence $m_1=\si_2(M_1)$ is a diagonal or an edge of $\si_2(H)$.
Since $m_1\le m$ and $m_1\le m_0$, we have that
either $m_0\le m$, or $m\le m_0, m\ne m_0$. By way of contradiction
assume that $m\le m_0, m\ne m_0$. However, then under some iteration of
$\si_2$ the leaf $m_0$, which is an eventual image of $m_1$, will be
mapped back to $\si_2(H)$ so that for the appropriate eventual image
$\si_2^j(m_0)$ of $m_0$ we have $\si_2^j(m_0)\le m\le m_0$, which is
only possible for the minor $m_0$ if in fact $m_0=m$, a contradiction.
Thus, $m_0\le m$, as desired.
\end{proof}

The next theorem describes some cases when one lamination tunes another one.
Recall that, by Definition~\ref{d:lamm}, the central gap of a lamination $\lam(m_0)$ is either a collapsing quadrilateral
or an infinite gap.

\begin{theorem}
  \label{t:com-pb}
Given minors $m_0$, $m_1$ and $m$, the following statements hold.

\begin{enumerate}

\item If $\lam(m_1)$ has majors $\pm M_1$ contained in the central gap of $\lam(m_0)$, then $\lam(m_0)\subset \lam(m_1)$;
if $m_1\ne m_0$, then $\lam(m_1)\ne \lam(m_0)$.

\item If $m$, $m_0$ and $m_1$ are non-degenerate minors such that $m_1$ is a child of $m_0$,
the lamination $\lam(m_1)$ coexists with $\lam(m)$,
and $m$ is neither $m_1$ nor an endpoint of $m_1$, then $\lam(m)\subset \lam(m_0)$.

\item If $m_1$ is an offspring of $m_0$ and $\lam(m)\subsetneqq \lam(m_1)$, then $\lam(m)\subsetneqq \lam(m_0)$.

\end{enumerate}

\end{theorem}

\begin{proof}
(1) Let the central gap $C(m_0)$ of $\lam(m_0)$ be a collapsing quadrilateral.
Then the fact that $\pm M_1\subset C(m_0)$ implies that $m_1=m_0$ and $\lam(m_1)=\lam(m_0)$.

Let now $C(m_0)$ be an infinite gap.
Then $m_0$ is periodic of the same period as $C(m_0)$.
Let us write $M_0$ for the pullback of $m_0$ that is invariant under the first return map $f$ of $C(m_0)\cap\uc$.
Assume that $M_1$ separates $-M_1$ from $M_0$ (or $M_1=-M_1$ is critical).
Consider iterated pullbacks of $M_1$ chosen so that each next pullback separates the previous pullback from $M_0$.
By definition of $m_1$-pullbacks, all these pullbacks belong to $\lam(m_1)$.
Since these $f$-pullbacks converge to $M_0$, we have $M_0\in \lam(m_1)$.
Similarly, all edges of $C(m_0)$ are in fact $m_1$-pullbacks of $M_0$,
 which implies that all edges of $C(m_0)$ belong to $\lam(m_1)$.
In the same way, it follows from definition of $m_1$-pullbacks that all other leaves of $\lam(m_0)$ are in fact leaves of $\lam(m_1)$. Hence, $\lam(m_0)\subset \lam(m_1)$.

(2) Let $\pm M$, $\pm M_i$ be the majors of $\lam(m)$, $\lam(m_i)$, for
$i=0$, $1$. Since the ``trivial'' cases of Lemma
\ref{l:tune-m1-le-m} do not hold, then by Lemma \ref{l:tune-m1-le-m} we
see that $m_1<m_0\le m$. Thus, $S(m_1)\subset S(m_0)\subset S(m)$.
Since $m_1$ maps to either $M_0$ or $-M_0$ under some iterate of $\si_2$
(see Proposition~\ref{p:der-min}), the majors $\pm M_0$ coexist with $\lam(m)$.
We have $\pm M_0\subset S(m_0)\subset S(m)$, therefore, $\pm M_0$ are contained in the central gap $C(m)$ of $\lam(m)$.
The result now follows from (1).

(3) By Theorem A we may assume that $m_1<m_{(n-1)/n}<\dots<m_{1/n}<m_0$
 where $m_{(i+1)/n}$ is a child of $m_{i/n}$ for $i=0$, $\dots$, $n-1$.
Applying (2) inductively, we see that $\lam(m)\subsetneqq \lam(m_{(n-1)/2})$, $\dots$, $\lam(m)\subsetneqq \lam(m_0)$.
\end{proof}

\section{Almost non-renormalizable minors: proof of Theorem B}

We begin by discussing which minors can be approximated by
offsprings of a given minor. Recall the following fact.

\begin{lemma}[\cite{thu85}, Lemma II.6.10a, part (b)]
  \label{l:nocross-m0}
  Let $m_0$ be a non-degenerate minor.
  If $m\le m_0$ is a minor, then $m_0\in\lam(m)$.
  In particular, $\si_2^n(m)$ cannot cross $m_0$ for $n>0$.
\end{lemma}

The next lemma elaborates on Lemma \ref{l:find-min}.

\begin{lemma}\label{l:minor-child}
Suppose that $\tilde m<m$ are two minors and $\si_2^n(\tilde m)=m$ for a minimal $n>0$. Then the following holds.
\begin{enumerate}
\item
If no image $\si_2^i(\tilde m)$ for $0<i<n$ is a minor separating $\tilde m$ from $m$, then $\tilde m$ is a child
of $m$ (in particular, $\tilde m\in \lam (m)$).
\item Let $\tilde m=m_0<m_1<\dots<m_{r-1}<m_r=m$ be all images of $\tilde m$ that are minors separating $\tilde m$
from $m$. Then $m_i$ is a child of $m_{i+1}$ for $0\le i\le r-1$.
\end{enumerate}
\end{lemma}

\begin{proof}
(1) To prove that $\tilde m\in \lam(m)$, consider $\si^i_2(\tilde m)$ for $0\le i\le n-1$.
Choose the greatest $i<n$ such that $\si_2^i(\tilde m)=m'$ satisfies $\tilde m\le m'\le m$.
Then no iterated image of $m'$ separates $\tilde m$ from $m$.
We claim that no image of $m'$ enters $S(m)$ vertically.
Indeed, otherwise the next image of $m'$ would have to enter $C(m)$ either separating $\tilde m$ and $m$
(impossible by the choice of $i$), or behind $\tilde m$ (impossible because $\tilde m$ is a minor).
Hence the leaves $\si_2^j(\tilde m)$, where $j=n-1$, $n-2$, $\dots$, $i$ are pullbacks of $m$ in $\lam(m)$.
Thus, $m'\in \lam(m)$ and is, therefore, a from $m$ derived minor.
If $i>0$, then $m'$ is a minor separating $\tilde m$ from $m$, a contradiction with the assumptions of the lemma.
We must conclude that $i=0$ and $m'=\tilde m$, in particular, $\tilde m\in\lam(m)$.
By definition, it follows that $\tilde m$ is a child of $m$.

(2) Follows from (1) applied to pairs of minors $m_i<m_{i+1}$, where $0\le i\le r-1$.
\end{proof}

The following lemma relates approximation by dynamical pullbacks and approximation by parameter pullbacks.

\begin{lemma}
  \label{l:ofs-approx}
  Let $m_0$ be a non-degenerate minor.
Suppose that $m\le m_0$ is a minor approximated by pullbacks of $m_0$ in $\lam(m)$.
Then $m$ can be approximated by offsprings of $m_0$.
\end{lemma}

\begin{proof}
We may assume that $m$ is never mapped to $m_0$ under $\si_2$.
By Lemma \ref{l:nocross-m0}, the chord $m_0$ is a leaf of $\lam(m)$.
Let $\ell_n$ be a sequence of leaves of $\lam(m)$ converging to $m$ and
such that $\si_2^{k_n}(\ell_n)=m_0$ for some $k_n$. Since infinitely
many $\ell_n$'s cannot share an endpoint with $m$, then we may assume
that all $\ell_n$ are disjoint from $m$ in $\ol\disk$.
We may assume that $\ell_n<m_0$.
If $\ell_n<m$ for infinitely many values of $n$, then, by Lemma~\ref{l:new-minor-1}, we may assume that these $\ell_n$ are minors,
and, by Lemma~~\ref{l:minor-child} and Theorem A, they are offsprings of $m_0$.
Suppose now that $\ell_n>m$ for infinitely many values of $n$; we may assume this is true for all $n$.
Consider all images of $\ell_n$ that separate
$m$ from $m_0$ and choose among them the closest to $m$ leaf
$\si_2^i(\ell)$. By Lemma~\ref{l:between} $\si_2^i(\ell)$ is a minor,
and by Theorem A $\si_2^i(\ell)$ is an offspring of $m_0$. This
completes the proof of the lemma.
\end{proof}

We can now prove the following theorem.

\begin{theorem}
  \label{t:R-approx}
  Let $m_0$ be a periodic non-degenerate minor, and let $m\le m_0$ be a non-degenerate minor.
Suppose that any lamination $\lam\subsetneqq\lam(m)$ satisfies $\lam\subsetneqq\lam(m_0)$.
Then $m$ is a limit of offsprings of $m_0$.
\end{theorem}

\begin{proof}
  By Lemma \ref{l:ofs-approx}, it suffices to approximate $m$ by
pullbacks of $m_0$ in $\lam(m)$. Consider the lamination $\lam_1$
consisting of iterated pullbacks of $m_0$ in $\lam(m)$ and their limits
(this includes the iterated images of $m_0$ since $m_0$ is periodic); then
$\lam_1\subset \lam(m)$. If $\lam_1=\lam(m)$, we are done; let
$\lam_1\ne\lam(m)$. Then, by our assumption, $\lam_1\subsetneqq
\lam(m_0)$. However, since $m_0\in\lam_1$, it follows from
Theorem~\ref{t:major} that $\lam_1=\lam(m_0)$, a contradiction with our
assumption.
\end{proof}

We need a lemma dealing with tuning of q-laminations.

\begin{lemma}\label{l:tunq}
Let $\lam_1\subsetneqq \lam_2$ be q-laminations where $\lam_1$ is not the empty lamination. Then $\lam_1$ has a periodic
quadratic Fatou gap, and, therefore, its minor is periodic and non-degenerate.
\end{lemma}

\begin{proof}
Suppose that $\lam_1$ does not have a periodic quadratic Fatou gap.
Then all gaps of $\lam_1$ are either (a) finite, or (b) infinite
eventually mapped to a periodic \emph{Siegel} gap for whom the first
return map is semiconjugate to an irrational rotation (the
semiconjugacy collapses the edges of the gap). Evidently, no leaves of
$\lam_2$ can be contained in finite gaps of $\lam_1$ because both
laminations are q-laminations. On the other hand, no leaves of $\lam_2$
can be contained in periodic Siegel gaps because any such leaf would
cross itself under a suitable power of $\si_2$ (this conclusion easily
follows from the semiconjugacy with an irrational rotation). Thus,
if $\lam_1$ does not have a periodic quadratic Fatou gap then no new leaves can be added
to $\lam_1$ and the inclusion $\lam_1\subsetneqq \lam_2$ is impossible.
\end{proof}

Recall that a quadratic lamination $\lam$ is called \emph{almost non-renormalizable} if
$\lam'\subsetneqq\lam$ implies that $\lam'$ is the empty lamination.
Note that all almost non-renormalizable laminations with non-degenerate
minors are q-laminations (if $\lam$ is not a q-lamination with a
non-degenerate minor then by Theorem \ref{t:major} there exists a
unique non-empty q-lamination $\hlam\subsetneqq \lam$, a
contradiction). The role of almost non-renormalizable minors is clear from the next lemma.

\begin{lemma}\label{l:ancestor}
Let $\lam$ be a lamination with non-degenerate minor $m$. Then there exists a unique almost
non-renormalizable lamination $\lam_0\subset \lam$ with non-degenerate minor $m_0$ such that
$m\subset \si_2(C(m_0))$.
\end{lemma}

\begin{proof}
Consider a lamination $\lam'\subset \lam$ with minor $m'$. Then, by
definition, $m\le m'$. Hence minors of all laminations contained in
$\lam$ are linearly ordered. Take the intersection $\lam_0$ of all
non-empty laminations contained in $\lam$; note that this intersection
is not the empty lamination as every non-empty lamination contains a
leaf of length at least $\frac13$. It follows that $\lam_0$ is itself a
non-empty lamination and that the minor $m_0$ of $\lam_0$ is such that
$m'\le m_0$ for every non-empty lamination $\lam'\subset \lam$ (here
$m'$ is the minor of $\lam'$). Evidently, $m\subset \si_2(C(m_0))$
(notice that if $\lam'\subset \lam$ then $C(\lam')\supset C(\lam)$).
\end{proof}

The set $\qml^{nr}$ by definition consists of all singletons in $\uc$ and the
postcritical sets of all almost non-renormalizable laminations.
The following theorem was obtained \cite{bot17}; 
for completeness, we prove it below.

\begin{theorem}
  \label{t:qml-nr}
  The set $\qml^{nr}$ is a lamination.
\end{theorem}

\begin{proof}
We only need to prove that $\qml^{nr}$ is closed in the Hausdorff metric.
We claim that $\qml^{nr}$ is obtained from $\qml$ by removing all minors that
are contained in the interiors of the gaps $\si_2(C(m))$ (except for their endpoints), where $m$ are non-degenerate almost non-renormalizable periodic minors.
The theorem will follow from this claim (indeed, the set of removed leaves is open in the Hausdorff metric).

Firstly, we show that a leaf $\ell$ of $\qml^{nr}$ cannot intersect the interior of a gap $G=\si_2(C(m))$ with $m\in\qml$.
Indeed, otherwise the fact that all our leaves are leaves of $\qml$ implies that $\ell\subset G$.
Hence the majors $\pm L$ of $\lam(\ell)$ are contained in $C(m)$.
By Theorem~\ref{t:com-pb}, part (1), we have then $\lam(m)\subsetneqq \lam(\ell)$.
By definition, this contradicts the fact that $\ell$ is a minor of an almost non-renormalizable lamination.

Secondly, suppose that $\tilde m$ is a minor that does not intersect
the interior of any gap $\si_2(C(m))$, where $m$ is a non-degenerate
periodic almost non-renormalizable minor. We may assume that $\tilde m$
is non-degenerate. We claim that $\tilde m\in\qml^{nr}$, i.e. that
$\tilde m$ is an edge of the postcritical set of an almost
non-renormalizable lamination. By way of contradiction, assume
otherwise. Observe that $\tilde m$ is an edge of the postcritical set
of the q-lamination $\hlam(\tilde m)$. By the assumption, it follows
that $\hlam(\tilde m)$ is not almost non-renormalizable. Hence by Lemmas \ref{l:tunq}
and \ref{l:ancestor} there exists a non-empty almost non-renormalizable lamination
$\lam'$ such that $\tilde m\subset \si_2(C(\lam'))$, a contradiction with the assumption
on $\tilde m$.

\end{proof}

Let $m_0$ be a non-degenerate periodic minor. Define the set $\ofl(m_0)$
consisting of all offsprings of $m_0$ and their limits.
The following theorem is a reformulation of Theorem B.

\begin{theorem}
  \label{t:thmC}
  The lamination $\qml^{nr}$ is the union of $\ofl(m_0)$, where $m_0$ runs through all edges of $\ca^c$.
\end{theorem}

\begin{proof}
Consider an almost non-renormalizable minor $m\in\qml^{nr}$.
There is an edge $m_0$ of the combinatorial main cardioid such that $m\le m_0$.
We claim that $m\in\ofl(m_0)$.
Indeed, consider all pullbacks of $m_0$ in $\lam(m)$ and all limit leaves of such pullbacks.
By \cite{bmov13}, this collection $\lam'$ of leaves is a lamination, and by construction $\lam'\subset \lam(m)$.
Since $\lam(m)$ is almost non-renormalizable, $\lam'=\lam$. Hence, pullbacks of $m_0$ in $\lam(m)$ approximate
$m$. By Lemma~\ref{l:ofs-approx}, the minor $m$ is approximated by offsprings of $m_0$.

Now, let $m\in\ofl(m_0)$, where $m_0$ is an edge of $\ca^c$.
Then there is a sequence of minors $\ell_i$ converging to $m$ such that each $\ell_i$ is an offspring of $m_0$.
We claim that $m$ is almost non-renormalizable, i.e., that $m\in\qml^{nr}$.
Assume the contrary: $m$ is contained in a gap $U$ of $\qml^{nr}$ and intersects the interior of $U$.
The only way it can happen is when $U=\si_2(C(m_1))$ is the postcritical gap of an almost
non-renormalizable lamination $\lam(m_1)$.
Then $\ell_i$ must also intersect the interior of $U$ for some $i$, hence $\ell_i$ must be contained in $U$.
By Theorem~\ref{t:com-pb}, part (1), we have $\lam(m_1)\subsetneqq \lam(\ell_i)$.
Since $\ell_i$ is an offspring of $m_0$, it follows by Theorem~\ref{t:com-pb}, part (3),
 that $\lam(m_1)\subsetneqq \lam(m_0)$.
However, this is impossible because $m_0$ itself is almost non-renormalizable,
 and the only lamination strictly contained in $\lam(m_0)$ is the empty lamination.
\end{proof}

\subsection*{Acknowledgments}
The first named author started discussing combinatorial models
for parameter spaces of cubic polynomials with the third named author
while visiting Jacobs University in Bremen in 2009. We are pleased to
thank Jacobs University for their hospitality. The second named author
was partially supported by NSF grant DMS--1807558. The third named
author was partially supported within the framework of the HSE
University Basic Research Program and the Russian Academic Excellence
Project `5-100'.

The authors also would like to thank the referees for valuable remarks
and suggestions.

\bibliographystyle{amsalpha}

\end{document}